\documentclass[reqno,a4paper]{amsart}

\usepackage[english,activeacute]{babel}
\usepackage{amssymb,amsmath,amsthm,amsfonts,mathrsfs,hyperref,mathtools,stmaryrd,bbm,xcolor}
\usepackage[cal=boondoxo]{mathalpha}
\usepackage[font=footnotesize]{caption}
\usepackage{tikz}
\usetikzlibrary{graphs,quotes,automata,intersections,shapes,trees,arrows,spy,positioning,decorations,arrows.meta,decorations.pathreplacing,patterns,calc}
\usepackage{tikz-cd}
\usepackage{tkz-euclide}
\usepackage{pgfplots}
\pgfplotsset{compat=newest}

\usepackage{extarrows}
\usepackage{graphicx}
\usepackage{epsfig,psfrag}
\usepackage{subfig}
\usepackage{rotating}
\usepackage{sidecap}
\usepackage{caption}
\usepackage{color}
\usepackage{cancel}
\usepackage{comment}
\DeclareMathAlphabet{\mathbbu}{U}{bbold}{m}{n}
\newcommand{\dt}{\mathrm{d}t}
\newcommand{\one}{\mathbbu{1}}
\usepackage[numbers,square,sort]{natbib}
\usepackage{tabularx}
\usepackage{enumitem} 
\usepackage[T1]{fontenc}
\usepackage[foot]{amsaddr}
\usepackage{url}
\usepackage{dsfont}
\usepackage{multicol}
\RequirePackage[normalem]{ulem}

\theoremstyle{plain}
\newtheorem{theorem}{Theorem}[section]
\newtheorem{lemma}[theorem]{Lemma}
\newtheorem{proposition}[theorem]{Proposition}
\newtheorem{corollary}[theorem]{Corollary}
\theoremstyle{definition}

\newtheorem{definition}[theorem]{Definition}

\theoremstyle{remark}
\newtheorem{remark}[theorem]{Remark}
\AtEndEnvironment{remark}{\popQED\(\diamondsuit\)}
\numberwithin{equation}{section}

\newcommand{\s}{\sigma}
\renewcommand{\l}{\lambda}
\renewcommand{\t}{\vartheta}
\renewcommand{\tt}{\vartheta\nu_1}
\newcommand{\D}{\Delta}
\newcommand{\sumi}{\sum_{i=1}^\infty}
\newcommand{\limi}{\lim_{i\to\infty}}
\newcommand{\E} {{\mathbb E}}
\renewcommand{\P} {{\mathbb P}}
\newcommand{\N} {{\mathbb N}}
\newcommand{\Z} {{\mathbb Z}}

\newcommand{\pn} {{(n)}}
\newcommand{\cL} {{\mathcal L}}
\newcommand{\cR} {{\mathcal R}}
\newcommand{\cY} {{\mathcal Y}}
\newcommand{\cZ} {{Z}}
\newcommand{\ca} {{a}}
\newcommand{\cS}{\mathcal{S}}

\newcommand{\Edelta}{E_{}^\Delta}
\newcommand{\deq}{\mathrel{\mathop:}=}

\definecolor{cl}{rgb}{0.0, 0.5, 0}
\definecolor{cmu}{rgb}{0.82, 0.1, 0.26}
\definecolor{ck}{rgb}{0,0.2,0.4}

\newcommand{\dd}{{\rm d}}

\newcommand{\ts}{\hspace{0.5pt}}
\newcommand{\myfrac}[2]{\frac{\raisebox{-2pt}{$#1$}}{\raisebox{0.5pt}{$#2$}}}

\newcommand*{\affmark}[1][*]{\textsuperscript{#1}}

\title[Killing versus catastrophes in birth-death processes]{Killing versus catastrophes in birth-death processes\\ and an application to population genetics}
\author{E. Baake\affmark[1] \and F. Cordero\affmark[1,2] \and E. Di Gaspero\affmark[1] \and A. Wakolbinger\affmark[3]}
\address{\newline\affmark[1]Faculty of Technology, Bielefeld University, Box 100131, 33501 Bielefeld, Germany\newline
\affmark[2]BOKU University, Institute of Mathematics, Department of Natural Sciences and Sustainable Ressources, Gregor-Mendel-Strasse 33/II, 1180 Vienna, Austria \newline
\affmark[3]Institut f\"ur Mathematik, Goethe-Universit\"at Frankfurt, 60629 Frankfurt am Main, Germany}
\email{ebaake@techfak.uni-bielefeld.de}
\email{fernando.cordero@boku.ac.at}
\email{edigaspero@techfak.uni-bielefeld.de}\email{wakolbinger@math.uni-frankfurt.de} 
\date{\today}

\begin{document}

\begin{abstract}
We establish connections between the absorption probabilities of a class of birth-death processes with killing, and the stationary tail of a related class of birth-death processes with catastrophes. The major ingredients of the proofs are  a decomposition of the dynamics of these processes, a Feynman--Kac type relationship for Markov chains with reset and rebirth, and the concept of Siegmund duality, which allows us to invert the relationship between the processes.\\
We apply our results to a pair of ancestral processes in population genetics, namely the killed ancestral selection graph and the pruned lookdown ancestral selection graph, in a finite population setting and its diffusion limit.
\end{abstract}
\maketitle
\bigskip
\noindent \emph{keywords:} birth-death processes; Markov chains with reset and rebirth; Siegmund duality; absorption and stationary probabilities;  Moran model; Wright--Fisher diffusion; ancestral selection graph.

\bigskip


\tableofcontents
\section{Introduction and main result}\label{S1}
Continuous-time birth-death processes appear in a large variety of contexts, from population genetics to demography, from epidemiology to queueing theory and many more. They are continuous-time Markov processes that describe the size of a population of individuals that can give birth (so that the state variable increases by one) or die (so that the state variable decreases by one). Two generalisations are the \textit{birth-death process with killing} and the \textit{birth-death process with catastrophes}. In the first case, the idea is to model a population where, regardless of the number of individuals present, the sudden extinction of the entire population is possible. The killed state is absorbing; it may either be a specific cemetery state, say \(\D\) (as in \cite{bd-killing1}), or it may coincide with \(0\), as in \cite{bd-killing2}. Birth-death processes with catastrophes on the other hand, see for example \cite{Brockwell85,Pollett01, Pollet03} (and \cite{KSN75} for a similar model in the framework of branching processes), have been used to study the dynamics of populations subject to catastrophes due to either death or large emigration events. In this case, the state may decrease by a range of values and may even move to the absorbing state \(0\), that is, the population dies out instantly. We will consider here a special case where catastrophes can induce losses of arbitrary size, but never extinguish the entire population. We will refer to it as \textit{birth-death process with non-killing catastrophes} (as for the birth-death processes with killing, the term \textit{killing} has to be understood here in terms of the entire population rather than single individuals); for brevity, will suppress the attribute `non-killing' throughout. An example is the \textit{pruned lookdown ancestral selection graph} (pLD-ASG), which was introduced to study genealogical structures in a model of population genetics \cite{LKBW15,JSP}.
\smallskip

A main motivation for this work was to shed light on an interesting relation between the absorption probabilities of $X$ and the equilibrium weights of \(\cZ\) for certain pairs $(X,Z)$, with $X$ a birth-death process with killing and $Z$ a birth-death process with catastrophes. We now introduce a few abbreviations, also with the purpose of unifying notation for the state spaces. We denote by \(\N_\infty\) the set \(\N\cup\left\{\infty\right\}\). For \(N\in\N_\infty\), the symbol \([N]\)~means the set \(\{ n \in \N : n\leq N\}\) and \([N]_0\) stands for \([N]\cup\{0\}\); in particular \([\infty]= \N\) and \([\infty]_0=\N_0\). Furthermore, for \(n,m \in \Z\), \([n:m]\coloneqq \{i \in \Z : n\leq i \leq m\}\). Let us further denote by \(\D\) an isolated cemetery state and let \([N]^\D\coloneqq [N]\cup\{\D\}\) as well as \([N]_0^\D\coloneqq [N]_0\cup\{\D\}\) for any \(N\in\N_\infty\). Notice that we allow \(N=\infty\), that is, \([N]^\D\) can be \(\N\cup\left\{\D\right\}\) and likewise for \([N]_0^\D\).

For a prescribed $N\in \mathbb N_\infty$, we now define a \emph{birth-death process with killing (or \emph{bdk} for short)} \(X:= X^N\coloneqq(X^N_t)_{t\geq 0}\) as the continuous-time Markov chain with state space \([N]_0^\D\) with jump rates
\begin{equation}\label{X_rates}
	\begin{aligned}
		q^{}_X(i,i+1) & = i\l_i, &\quad & i\in [N-1],\\
		q^{}_X(i,i-1) & = i\mu_i, &\quad &i\in [N],\quad \text{and } \\
		q^{}_X(i,\D) &= i\kappa, &\quad &i\in [N],
	\end{aligned}
\end{equation}
where \(\kappa>0\), \(\l_i > 0\) for \(i\in[N-1]\), and \(\mu_i > 0\) for \(i\in [N]\); here and in what follows, we tacitly understand jump rates that are not mentioned to be zero. For later use, we complement this with \(\l_{N}=\mu_{N+1}=0\) when \(N<\infty\). See Figure~\ref{Fig:XcZ} (top) for the transition graph and note that the extra state $\D$ is included to distinguish between the effects of killing (via parameter $\kappa$) and absorption in 0 through successive death events (via the $\mu_i$). 

\begin{figure}[b]
	\centering
	\begin{tikzpicture}[->,shorten >=1pt,auto,node distance=3cm,semithick]
		\tikzstyle{every state}=[draw=black,text=black]
		
		\node[state,minimum size=1.2cm] (0X) {0};
		\node[state,minimum size=1.2cm] (1X) [right=1.5cm of 0X] {1};
		\node[state,minimum size=1.2cm] (2X) [right=1.5cm of 1X]{2};
		
		\node							(dotsX) [right=1cm of 2X] {\(\cdots\)};
		\node[state,minimum size=1.2cm] (4X) [right=1cm of dotsX] {\(N\!-\!1\)};
		\node[state,minimum size=1.2cm] (5X) [right=1.5cm of 4X] {\(N\)};

		\node (d1X) [above=1.5cm of 1X] {\(\D\)};
		\node (d2X) [above=1.5cm of 2X] {\(\D\)};
		\node (d4X) [above=1.5cm of 4X] {\(\D\)};
		\node (d5X) [above=1.5cm of 5X] {\(\D\)};
		
		\node[state,minimum size=1.4cm,opacity=0] (downX) [below = -0.3cm of dotsX] {};
		\node[state,minimum size=1.4cm,opacity=0] (upX) [above = -0.3cm of dotsX] {};
		
		\path
		(1X) edge	[cmu]				node {\(\mu_1\)}				(0X)
		(1X) edge 	[left,ck]			node {\(\kappa\)}				(d1X)
		(2X) edge	[left,ck]			node {\(2\kappa\)}				(d2X)
		(4X) edge	[left,ck]			node {\((N-1)\kappa\)}			(d4X)
		(5X) edge	[left,ck]			node {\(N\kappa\)}				(d5X)
		(2X) edge	[bend left,cmu]		node {\(2\mu_2\)}				(1X)
		(1X) edge	[bend left,cl]		node {\(\l_1\)}					(2X)
		(4X) edge	[bend left,cl]		node {\((N-1)\l_{N-1}\)}		(5X)
		(5X) edge	[bend left,cmu]		node {\(N\mu_N\)}				(4X)
		
		(downX) edge	[out=180,in=315,cmu]	node{} 				(2X)
		(2X) edge		[in=180,out=45,cl]		node{} 				(upX)
		(4X) edge		[in=0,out=225,cmu]		node{} 				(downX)
		(upX) edge		[in=135,out=0,cl]		node{} 				(4X)
		;

	
	 	\node[state,minimum size=1.2cm] (1) [below=3cm of 0X] {1};
		\node[state,minimum size=1.2cm] (2) [below=3cm of 1X] {2};
		\node[state,minimum size=1.2cm] (3) [below=3cm of 2X] {\(3\)};
		\node							(dots) [right=1cm of 3] {\(\cdots\)};
		\node[state,minimum size=1.2cm] (4) [below=3cm of 4X] {\(N\!-\!2\)};
		\node[state,minimum size=1.2cm] (5) [below=3cm of 5X] {\(N\!-\!1\)};
		
		\node[state,white,minimum size=1.4cm,opacity=0] (down) [below = -0.3cm of dots] {};
		\node[state,white,minimum size=1.4cm,opacity=0] (up) [above = 1.4cm of dots] {};
		\node[state,white,minimum size=1.4cm,opacity=0] (center) [right = 0.7cm of 3] {};
		
		\path 
		(1)		edge	[cl]					node {\(\l_{2}\)}				(2)
		(2)		edge	[bend left,cmu]			node {\(\mu_{3}\)}				(1)
	 	(2)		edge	[cl]					node {\(2\l_{3}\)}				(3)
		(3)		edge	[bend left,cmu]			node {\(2\mu_{4}\)}				(2)
		(4)		edge	[cl]					node {\((\!N\!-\!2\!)\l_{\!N\!-\!1}\)}	(5)
		(5)		edge	[bend left,cmu]			node {\((N\!-\!2)\mu_{N}\)}		(4)

		(2)		edge	[bend right=70,ck]		node {\(\kappa\)}				(1)
		(3)		edge	[bend right=70,ck]		node {\(\kappa\)}				(1)
		(3)		edge	[bend right=70,ck]		node {\(\kappa\)}				(2)
		(5) 	edge	[bend right=70,ck]		node {\(\kappa\)}				(4)
		
		(down)	edge	[cmu,in=315,out=180]	node {}							(3)
		(3)		edge	[cl]					node {}							(center)		
		(up)	edge	[ck,out=180,in=90]		node {}							(3)
	 	(up)	edge	[ck,out=180,in=90]		node {}							(2)
		(up)	edge	[ck,out=180,in=70]		node {}							(1)
		
		(center)	edge	[cl,out=0,in=180]	node {}							(4)
		(4)		edge	[ck,out=90,in=0]		node {}							(up)
		(5)		edge	[ck,out=90,in=0]		node {\(\kappa\)}				(up)
		(4)		edge	[cmu,out=225,in=0]		node {}							(down)
		;
	\end{tikzpicture}
	\caption{The transition graphs of \(X^{N}\) (top) and \(\cZ^{N-1}\) (bottom) for finite \(N\) (see \eqref{X_rates}, \eqref{calZ_rates}, and Definition~\ref{def:paired}).
	\label{Fig:XcZ}}
\end{figure}

A \emph{birth-death process with catastrophes} on $[N-1]$ (or \emph{bdc} for short) is a process $Z\coloneqq Z^{N-1}\coloneqq (Z^{N-1}_t)_{t\geq 0}$ with jump rates
\begin{equation}\label{calZ_rates}
	\begin{aligned}
		q^{}_{\cZ}(i,i+1) & = i\lambda_{i+1}, &\quad& i\in [N-2], \\
		q^{}_{\cZ}(i,i-1) & =(i-1)\mu_{i+1}+\kappa, &\quad& i\in [2: N-1], \quad \text{and } \\
		q^{}_{\cZ}(i,j) &= \kappa \quad \mbox{ from } i \mbox{ to }j &\quad& \text{ for } i \in [3:N], j \in [i-2],
	\end{aligned}
\end{equation}
with $(\mu_i)$, $(\lambda_i)$, and $\kappa$ as in \eqref{X_rates}, see Figure~\ref{Fig:XcZ} (bottom). (We sometimes write $(\mu_i), (\lambda_i)$ and the like for sequences (or vectors) without specifying their ranges when there is no risk of confusion.) Let us anticipate that~condition~\eqref{cond_lambda} will prevent $Z$ from exploding.

In a wider sense, we will also speak of a process as a bdc if it assumes the form~\eqref{X_rates} after relabelling the states, and likewise for a bdk. Note that, individually for $X$ and $Z$, respectively, \eqref{X_rates} and \eqref{calZ_rates} leave the birth and death rates perfectly general (apart from positivity, which guarantees irreducibility of $Z$ also in the case $\kappa=0$). In contrast, we restrict ourselves to \emph{homogeneous} killing and catastrophes in the sense that $\kappa$ is constant. 

The parametrisation in \eqref{calZ_rates}, which (as discussed in the paragraph after the next) may seem a little strange at first sight, was chosen in view of the following definition and our main result, Theorem~\ref{thm_general}.

\begin{definition}\label{def:paired}
We say that a \emph{birth-death process with killing} $X=X^N$ with state space $[N]_0^\Delta$ and a \emph{birth-death process with catastrophes} $Z=Z^{N-1}$ with state space $[N-1]$ are \emph{paired} if their rates can be parametrised in the form~\eqref{X_rates} and~\eqref{calZ_rates} with \emph{the same} parameter set $(\lambda_i)$, $(\mu_i)$, and $\kappa$.
\end{definition}

The superscripts $N$ and $N-1$ are meant to remind ourselves of the state spaces (which are different for paired processes); unless stated otherwise, $(X,Z)$ stands for $(X^N, Z^{N-1})$. A note on the parametrisation of $q_{\cZ}$ is in order. For our application, or if $\cZ$ were considered in isolation, it would be more natural to replace $\lambda_{i+1}$ and $\mu_{i+1}$ by $\lambda_i$ and $\mu_i$, respectively, so that $\lambda_i$ and $\mu_i$ represent the per-capita birth and death rates of an individual in a population of size $i$ --- as in \eqref{X_rates}, except that, in $\cZ$, there is one immortal individual. Such a formulation is considered in \cite{Enricos-thesis}. In the current contribution, the parametrisation of $q_{\cZ}$ is chosen so as to streamline the mathematical treatment of the connection between $X$ and $\cZ$; for the application in Section~\ref{sec:app}, we will have to tweak the parameters a bit, at least in the finite-$N$ case.

We denote by \(b_i\) the absorption probability of \(X\) in \(0\) when starting from \(i\), that is,
\begin{equation}\label{eq:bnb}
	b_i = b_i^N\coloneqq\P(\text{\(X\) absorbs  in \(0\)}\mid X_0=i),\quad i\in [N]^\D_0,
\end{equation}
and summarise them into the vector $b \coloneqq (b_i)_{i \in [N]_0^\Delta}$.  Since \(\mu_i>0\) for every \(i\in[N]\), one has \(b_i>0\) for every \(i\in[N]_0\). In any case, a first-step decomposition of the absorption probabilities shows that the \(b_i\) satisfy the recursion
\begin{equation}\label{eq:rec_b_gen}
	(\l_i + \mu_i + \kappa)b_i = \l_i b_{i+1} + \mu_i b_{i-1} \quad \text{for }\; i\in [N-1],
\end{equation}
along with the boundary conditions \(b_0=1\) and \(b_\Delta=0\); we will see in Proposition~\ref{pSD} that, additionally, $\lim_{i \to \infty} b_i=0$ in the case $N=\infty$.

With $\cZ_{\text{eq}}$ the random equilibrium state of $\cZ$ (if an equilibrium distribution exists), denote the {\em stationary tail probabilities} by
\begin{equation}\label{tailprob}
	\ca_i = \ca_i^{N-1} \coloneqq \mathbb P(\cZ_{\text{eq}} > i), \quad i \in [N-1]_0,
\end{equation}
and summarise them into the vector $a \coloneqq (a_i)_{i \in [N-1]_0}$. The corresponding probability weights are $w_i = w_i^{N-1} \coloneqq \ca_{i-1}-\ca_i$ for $i \in [N-1]$. We will see in Corollary~\ref{cor:rec_a} that the \(\ca_i\) follow the recursion
\begin{equation}\label{eq:rec_a_gen}
	(\mu_{i+1}+\l_{i}+\kappa)\ca_i=\l_{i}\ca_{i-1}+\mu_{i+1}\ca_{i+1} \quad \text{for } \; i \in [N-2],
\end{equation}
complemented by the boundary conditions \(\ca_0=1\) together with \(\ca_{N-1}=0\) when \(N<\infty\), and \(\lim_{i\to\infty}\ca_i=0\) otherwise. Eq.~\eqref{eq:rec_a_gen} is a generalisation of (4) in \cite{LKBW15} and looks like the first-step decomposition of an absorption probability; indeed, it will emerge as such via the Siegmund dual of \(\cZ\).

As part of our main result, we will obtain, under conditions to be specified below, that
\begin{equation}\label{cor}
	\frac{b_i}{b_2} = \frac{\ca_{i-2}-\ca_{i-1}}{\ca_0-\ca_1}\prod_{j=1}^{i-2}\frac{\mu_{j+2}}{\lambda_{j+1}}, \quad i\in [2:N], 
\end{equation}
along with the complementary relationship
\begin{equation}\label{cor_complement}
	\ca_i= \frac{b_{i+1}-b_{i+2}}{b_1-b_2} \prod_{j=1}^i \frac{\lambda_{j+1}}{\mu_{j+1}}, \quad i\in [N-2]_0.
\end{equation}
These connections between $X$ and $\cZ$ are, at the same time, relations between the solutions of the recursions \eqref{eq:rec_a_gen} and \eqref{eq:rec_b_gen}.

The special case of $N=\infty$ with
\begin{equation}\label{special}
	\lambda_i := \sigma, \quad \mu_i := i-1 + \vartheta\nu_1, \quad \kappa := \vartheta \nu_0, \quad i \in \N_0,
\end{equation}
with $\sigma,\vartheta, \nu_0, \nu_1>0$ and $\nu_0 + \nu_1=1$, appears in the population genetic framework of a stationary two-type Wright--Fisher diffusion $\mathcal{Y}$ with two-way mutation at rate $\vartheta \nu_0$ from type~$1$ to type $0$, and rate $\vartheta \nu_1$ from type $0$ to type $1$, where type $0$ has selective advantage~$\sigma$. Here, $\mathcal Y$ describes the evolution of the proportion of type 1 individuals. One is specifically interested in \emph{genealogical questions}, such as the type of the {\em common ancestor}, that is, the type of the ancestor of the (entire) population in the remote past. In \cite{F02}, the weight~$p_1$ of this distribution in $1$ was computed (by mainly analytic methods). In \cite{LKBW15} and \cite{JSP}, the underlying genealogical structures were revealed, namely the killed ASG (kASG) and the pruned lookdown ASG (pLD-ASG). The kASG contains the sample’s potential ancestry back until the mutations that decide about the event that all individuals in the sample are of deleterious type, cf. \cite{JSP} (see also \cite{Sh-Uch}). The pLD-ASG contains all \emph{potential} ancestors of the population, subject to a pruning procedure that, upon mutation, eliminates lines that can never be ancestral, see \cite{LKBW15}. Notably, in the special case~\eqref{special}, the processes $X$ and $\cZ$ (with jump rates~\eqref{X_rates} and~\eqref{calZ_rates}) become the line-counting processes $\mathcal R$ and $\mathcal L$ of the kASG and of the pLD-ASG, respectively, and~\eqref{cor} as well as~\eqref{cor_complement} specialise immediately to relations between the hitting probabilities of the former and the stationary tail probabilities of the latter, see Section~\ref{diffusionlimit}. In a population with finite size $N$, the corresponding relationship is more intricate; this will be elaborated in Section~\ref{sec:Moran}.

Still in the case~\eqref{special}, the process $\mathcal R=X$ is in moment duality with the Wright--Fisher diffusion $\mathcal{Y}$, whence
\begin{equation}\label{moment_dual}
	\mathbb E[\mathcal{Y}_{\text{eq}}^i] = \beta_i,
\end{equation}
where $\mathcal{Y}_{\text{eq}}$ is the (random) equilibrium state of $\mathcal{Y}$ and $\beta_i:= b_i^{\infty}$ is as in~\eqref{eq:bnb}. In \cite{LKBW15} and \cite{JSP}, it was found that 
\begin{equation}\label{repp1}
	p_1 = \sum_i \omega_i\, \beta_i
\end{equation} 
with the $\omega_i = w_i^{\infty} $ as defined below~\eqref{tailprob} with $\mathcal L:= Z$ in case~\eqref{special} ; namely, in the light of \eqref{moment_dual}, $p_1$ is the probability that all individuals in a sample of random size $\mathcal L_{\text{eq}}$ drawn from the equilibrium Wright--Fisher population are all of type 1. 

While the connection between the  kASG and the pLD-ASG reflected in \eqref{repp1} is well understood and has been extended to a larger class of models (\cite{BCH18,BCH22,CV23,luigi,Vechambre23}, to name just a few), the connection \eqref{cor_complement} is less well known. It appears in \cite{Taylor_07} for $i=1$ and in \cite{BCdG24} for all $i$, and is used there to characterise the mutation process along the ancestral line; in both instances, it was proved by analytical means, and its probabilistic content has, so far, remained unclear. The companion relationship \eqref{cor} has, to the best of our knowledge, not appeared in the literature yet.
In the present work, we will give probabilistic proofs of~\eqref{cor} and \eqref{cor_complement} for {\em any} pair $X$, $\cZ$ with the jump rates $q^{}_X$ and $q^{}_{\cZ}$ given by~\eqref{X_rates} and~\eqref{calZ_rates}. 

\subsection{Main result}\label{Sec:results}
As already announced, the parameter $N$ can be either a natural number, or stand for $\infty$. The following two conditions will be relevant only in the case $N=\infty$: 

\begin{minipage}{6cm}
	\begin{align}\tag{1.12~a} \label{cond_lambda}
		\sum_{i=1}^\infty \frac{1}{\l_i} = \infty,
	\end{align}
\end{minipage}\hspace{3.5cm}
\begin{minipage}{6cm}
	\begin{align}\tag{1.12~b} \label{cond_mu}
		\sum_{i=1}^\infty \frac{1}{\mu_i} = \infty .
	\end{align}
\end{minipage}

\begin{theorem}\label{thm_general}
	For $N\in \mathbb N_\infty$, let $X$ and $\cZ$ be a birth-death process with killing and a birth-death process with catastrophes, respectively, paired in the sense of Definition~\ref{def:paired}.  In the case $N=\infty$, assume~\eqref{cond_lambda} and~\eqref{cond_mu}.
	\begin{itemize}
		\item[\bf A.] 
		Under these assumptions, $\cZ$ has an equilibrium distribution, and its stationary tail probabilities \((\ca_i)\)  (as given in \eqref{tailprob}) and the hitting probabilities \((b_i)\) of $X$ (as given in \eqref{eq:bnb}) obey the relationship~\eqref{cor}.
		\item[\bf B.]
		Furthermore, also the inverse relationship~\eqref{cor_complement} is valid.
	\end{itemize}
\end{theorem}

\subsection{Outline of the paper.}
The proof of part A of this theorem will be given in Section~\ref{proofpartA}. It relies on Proposition~\ref{CorrProp}, which relates certain hitting probabilities of $X$ to the stationary probability weights of $Z$ (and the $(\lambda_i)$ and $(\mu_i)$). Section~\ref{sechlrs} is devoted to more general \emph{Markov processes with reset and rebirth} and, in Proposition~\ref{keyprop}, presents a biased detailed balance relation. This result seems interesting in its own right, and indeed Section~\ref{sechlrs} can be read independently of the rest of the paper. The central concept in the proof of part~B of Theorem~\ref{thm_general} is \emph{Siegmund duality} (see \cite{Siegmund} or \cite{Jansen-Kurt}). We will use a pathwise construction that relies on a decomposition of the dynamics of $X$ and $\cZ$ in the spirit of Clifford--Sudbury flights. In Section~\ref{Sec:siegmund}, we will recall the concept and relevant facts from \cite{CS85} and \cite{ECP} and establish flight constructions for general birth-death processes with catastrophes and with killing, respectively, together with their Siegmund duals, and in Section~\ref{sec:proof_part_B}, the constructions are used to prove part B of Theorem~\ref{thm_general}. The strategy relies on the well-known property of Siegmund duality to turn the absorption probabilities of a birth-death process into stationary probabilities of its dual (see \cite{Siegmund,Cox-Rösler,van Doorn}). We will see how this translates Part~A of Theorem~\ref{thm_general} into Part~B. Indeed, the duality allows us to express \((\ca_i)\), now with the meaning of the vector of absorption probabilities of \(\cZ^\star\), the Siegmund dual of $Z$, in terms of \( (b_i)\), now with the meaning of the stationary tail probabilities of \(X^\circ\), the inverse Siegmund dual of $X$; the latter is to say that $X$ is the Siegmund dual of $X^\circ$. Figure~\ref{Fig:intro} shows the relations between the four processes and the roles played by \(\ca\) and \(b\). This is just a coarse picture; there are some subtle details hidden, in particular behind the lower link between \(X^\circ\) and $Z^*$. Specifically, some shifting of state spaces is involved, because $X^\circ$ is a bdc in the wider sense only, and because it is not paired with $Z^*$ in the sense of Definition~\ref{def:paired}; this will be clarified later. 

\begin{figure}[t]
	\centering
	\begin{tikzpicture}[->,shorten >=1pt,auto,node distance=3cm,semithick]
		\node[draw,rectangle,rounded corners=0.4cm, fill=none,minimum width=6cm,minimum height=1.5cm] (X) {\shortstack{bdk \(X\)\\ absorption probabilities \(b\)}};
		\node[draw,rectangle,rounded corners=0.4cm, fill=none,minimum width=6cm,minimum height=1.5cm] (Z) [right=3.5cm of X] {\shortstack{bdc \(\cZ\)\\ stationary tail probabilities \(\ca\) }};
		
		\node[draw,rectangle,rounded corners=0.4cm, fill=none,minimum width=6cm,minimum height=1.5cm] (Xs) [below=3cm of X] {\shortstack{bdc \(X^\circ\)\\ stationary tail probabilities \(b\) }};
		\node[draw,rectangle,rounded corners=0.4cm, fill=none,minimum width=6cm,minimum height=1.5cm] (Zs) [right=3.5cm of Xs] {\shortstack{bdk \(\cZ^\star\)\\ absorption probabilities \(\ca\)}};
		
		\draw[->] (Z) -- node[above] {\(b\) in terms of \(\ca\)} node[below] {Theorem \ref{thm_general} A} (X);
		\draw[->] (Xs) -- node[left] {\shortstack{Siegmund dual\\Proposition \ref{pSD} (2)}} (X);
		\draw[->] (Z) -- node[right] {\shortstack{Siegmund dual\\Proposition \ref{pSD} (1)}} (Zs);
		\draw[->] (Xs) -- node [above] {\(\ca\) in terms of \(b\)} node[below] {Theorem \ref{thm_general} B} (Zs);
	\end{tikzpicture}
	\caption{The connections between the paired processes $X$ and $Z$ of Definition~\ref{def:paired}, the inverse Siegmund dual $X^\circ$ of $X$, and the Siegmund dual \(\cZ^\star\) of $Z$; $X^\circ$ is a bdc in the wider sense, as specified in the paragraph below \eqref{calZ_rates}. \label{Fig:intro}}
\end{figure}

In Section~\ref{sec:app}, we detail the application to genetics, more precisely to the so-called Moran model with selection and mutation (for finite $N$) and its $N \to \infty$ limit, the aforementioned Wright--Fisher diffusion.

\section{Proof of Theorem~\ref{thm_general}~A}\label{proofpartA}
For $n \in \N_0$, let $T^X_n$ be the first time at which the process $X$ hits the state $n$. By the strong Markov property, the l.h.s. of~\eqref{cor} equals
\begin{equation*}\label{lhs}
	\frac{b_i}{b_2} = \frac{\P_i(T_0^X<\infty)}{\P_2(T_0^X<\infty)} = \prod_{n=2}^{i-1}\P_{n+1}(T_{n}^X< \infty), \quad i\in [2:N]. 
\end{equation*}
The r.h.s. of~\eqref{cor} is
\begin{equation*}\label{rhs}
	\frac{\P(\cZ_{\text{eq}} = i-1)}{\P(\cZ_{\text{eq}} = 1)}\prod_{n=2}^{i-1} \frac{\mu_{n+1}}{\lambda_{n}}= \prod_{n=2}^{i-1} \left(\frac{\P(\cZ_{\text{eq}} = n)}{\P(\cZ_{\text{eq}} = n-1)}\cdot\frac{\mu_{n+1}}{\lambda_{n}}\right).
\end{equation*}
In view of this product structure, it is thus sufficient to check the equality ``factor by factor'', and Theorem~\ref{thm_general} Part A. is immediate from the following
\begin{proposition}\label{CorrProp}
	\begin{equation*}
		\P_{n+1}(T_n^X< \infty) = \frac
		{\P(\cZ_{\rm eq} = n)}
		{\P(\cZ_{\rm eq} = n-1)}
		\cdot
		\frac{\mu_{n+1}}{\lambda_n}, \quad n\in[2: N-1].
	\end{equation*}
\end{proposition}
The proof will be given at the end of this section, prepared by a series of lemmas. A key idea is to work with the family of processes $\cZ^{(n)}, n\in [N-2]$, with the jump rates
\begin{align}
	q^{(n)}_{\cZ}(i,i+1) &:= \lambda_{i+1}, \qquad \qquad \quad \quad \, \, \, i \in [n:N-2], \nonumber \\
	q^{(n)}_{\cZ}(i,i-1) &:= \mu_{i+1}, \qquad \qquad \quad \quad \,\,\,i \in [n+2:N-1], \label{qZn_rates}\\
	q^{(n)}_{\cZ}(i,n) &:= \kappa + \mu_{i+1} \mathbbm{1}_{\left\{i=n+1\right\}}, \quad i \in [n+1:N-1]; \nonumber
\end{align}
see Figure~\ref{Fig:Zsn} for the transition graph.

\begin{figure}[h!]
	\begin{tikzpicture}[->,shorten >=1pt,auto,node distance=3cm,semithick]
		\tikzstyle{every state}=[draw=black,text=black]
		
		\node[state,minimum size=1.2cm] (1)	{\(n\)};
		\node[state,minimum size=1.2cm] (2) [right=1.5cm of 1] {\(n+1\)};
		\node[state,minimum size=1.2cm] (3) [right=1.5cm of 2] {\(n+2\)};
		\node		(dots) [right=1cm of 3] {\(\cdots\)};
		\node[state,minimum size=1.2cm] (4) [right=1.5cm of dots] {\(N\!-\!2\)};
		\node[state,minimum size=1.2cm] (5) [right=1.5cm of 4] {\(N\!-\!1\)};
		
		\node[state,white,minimum size=1.4cm,opacity=0] (down) [below = -0.3cm of dots] {};
		\node[state,white,minimum size=1.4cm,opacity=0] (up) [above = 1.4cm of dots] {};
		\node[state,white,minimum size=1.4cm,opacity=0] (center) [right = 0.7cm of 3] {};
		
		\path 
		(1)		edge	[cl]					node {\(\l_{n+1}\)}				(2)
		(2)		edge	[bend left,cmu]			node {\(\mu_{n+2}\)}			(1)
		(2)		edge	[cl]					node {\(\l_{n+2}\)}				(3)
		(3)		edge	[bend left,cmu]			node {\(\mu_{n+3}\)}			(2)
		(4)		edge	[cl]					node {\(\l_{N-1}\)}				(5)
		(5)		edge	[bend left,cmu]			node {\(\mu_{N}\)}				(4)

		(2)		edge	[bend right=70,ck]		node {\(\kappa\)}				(1)
		(3)		edge	[bend right=70,ck]		node {\(\kappa\)}				(1)
			
		(down)	edge	[cmu,in=315,out=180]	node {}							(3)
		(3)		edge	[cl]					node {}							(center)
		(up)	edge	[ck,out=180,in=70]		node {}							(1)
		
		(center)	edge	[cl,out=0,in=180]	node {}							(4)		
		(4)		edge	[ck,out=90,in=0]		node {}							(up)
		(5)		edge	[ck,out=90,in=0]		node {\(\kappa\)}				(up)
		(4)		edge	[cmu,out=225,in=0]		node {}							(down)
		;
		
	\end{tikzpicture}
	\caption{The transition graph of \(\cZ^\pn\) for finite \(N\) (see \eqref{qZn_rates} and proof of Lemma~\ref{projectionlemma}).\label{Fig:Zsn}}
\end{figure}

\smallskip
Note that the dynamics of $Z$ may be decomposed into the dynamics of the $Z^{(n)}$ in the sense that, for \(i,j\in [N-1], i \neq j\), \
\begin{equation}\label{calzdecomp}
	q_{\cZ}(i,j)=\sum_{n:\, i,j\in [n:N-1]} q^{(n)}_{\cZ}(i,j),
\end{equation}
as is readily checked. Let $\cZ^{(n)}_{\text{eq}}$ be an $[n:N-1]$-valued random variable whose law is the equilibrium distribution of the process $\cZ^{(n)}$.
\begin{lemma}\label{projectionlemma}
	For $n=2, \ldots, N-1$,
	\begin{equation*}
		\frac{\P(\cZ_{\rm eq} = n)}{\P(\cZ_{\rm eq} = n-1)} = \frac{\P(\cZ^{(n-1)}_{\rm eq} = n)}{\P(\cZ^{(n-1)}_{\rm eq} = n-1)}.
	\end{equation*}
\end{lemma}
\begin{proof}
We abbreviate $w_i:= \P(Z_{\text{eq}}=i)$, $i\in [2:N-1]$, and claim that, for all $n\in[2:N]$, the weights $w_i$ and the probability weights of $Z_{\text{eq}}^{(n-1)}$ obey the proportionality relation
\begin{equation}\label{proprel}
	\frac{w_j}{w_{j'}} =\frac{\P(Z_{\text{eq}}^{(n-1)}=j)}{\P(Z_{\text{eq}}^{(n-1)}=j')}\, ,\quad j,j'\in [n-1:N-1].
\end{equation}
Obviously,~\eqref{proprel} specialises to Lemma~\ref{projectionlemma} by choosing $(j,j'):=(n, n-1)$. To prove~\eqref{proprel}, it suffices to check that the weights $w_i$, $i\in [n-1:N-1]$, satisfy the stationarity condition (or ``balance equations'') for the jump rates $q_Z^{(n)}$, that is, 
\begin{equation}\label{rec_wn}
	\lambda_{n}w_{n-1} = \mu_{n}w_{n}+\sum_{j\in [n:N-1]}\kappa w_j,
\end{equation}
and (with the conventions $w_{N} := 0$ and $\mu_{N+1} := 0)$,
\begin{equation}\label{balanceinner}
	(\lambda_{i+1}+\mu_{i+1}+\kappa)w_i = \lambda_{i}w_{i-1} + \mu_{i+2}w_{i+1}, \quad i \in [n:N-1].
\end{equation}
For the time-stationary process $Z=(Z_t)_{t \geq 0}$ and any subset $A$ of its state space, one has 
\begin{align}\label{flows}
	\begin{split}
		\P(Z_0\notin A, Z_t\in A)&=\P(Z_t\in A)- \P(Z_t\in A, Z_0\in A) \\&= \P(Z_0\in A)- \P(Z_0\in A, Z_t\in A)=\P(Z_0\in A, Z_t\notin A), \qquad t > 0,
	\end{split}
\end{align}
which is in line with the well-known fact that the stationary probability fluxes into and out of the set~$A$ are balanced. With $A:=[n:N-1]$, and as $t\to 0$, the l.h.s.~of~\eqref{flows} equals $w_{n-1}\lambda_n t+o(t)$, while the right-hand side of~\eqref{flows} equals $\bigl(w_{n}(n-1)\mu_{n+1} +\bigl(\sum_{j\in [n:N-1]}w_j\bigr)(n-1)\kappa\bigr)t +o(t)$ as $t\to 0$. This proves~\eqref{rec_wn}. The claimed equalities~\eqref{balanceinner} then follow by subtracting two copies of \eqref{rec_wn} from each other, one with $n:= i+1$ and the other with $n:= i$. 
\end{proof}
For $n \in [N-1]$, let the jump rates $q^{(n)}_X$ be defined as
\begin{equation}
	\begin{split} \label{qXn_rates}
		q^{(n)}_{X}(i,i+1) & \coloneqq \l_i, \quad i\in [n:N-1],\\
		q^{(n)}_{X}(i,i-1) & \coloneqq \mu_i, \quad i\in [n:N],\\
		q^{(n)}_{X}(i,\D) & \coloneqq \kappa, \quad i\in [n:N].
	\end{split}
\end{equation}
Let $X^{(n)}$ be a process that follows the jump rates $q_X^{(n)}$ (see Figure \ref{Fig:Xn} for the transition graph). Note that a decomposition of the dynamics of $X$ into the dynamics of the processes $X^{(n)}$ can be done in analogy with \eqref{calzdecomp}. However, the crucial point here is a relationship between $T^X$ and $T^{X^{(n)}}_n$, where the latter is the waiting time until $X^{(n)}$ hits $n$.

\begin{figure}[t]
	\begin{tikzpicture}[->,shorten >=1pt,auto,node distance=3cm,semithick]
		\tikzstyle{every state}=[draw=black,text=black]
		
		\node[state,minimum size=1.2cm] (1) [minimum size=33]{\(n-1\)};
		\node[state,minimum size=1.2cm] (2) [right=1.5cm of 1,minimum size=33] {\(n\)};
		\node[state,minimum size=1.2cm] (3) [right=1.5cm of 2,minimum size=33] {\(n+1\)};
		\node[state,minimum size=1.2cm] (4) [right=1.5cm of 3,minimum size=33] {\(n+2\)};
		\node	(dots) [right=1cm of 4] {\(\cdots\)};
		\node[state,minimum size=1.2cm] (5) [right=1cm of dots,minimum size=33] {\(N\)};
		\node (d2) [above=1.5cm of 2] {\(\D\)};
		\node (d3) [above=1.5cm of 3] {\(\D\)};
		\node (d4) [above=1.5cm of 4] {\(\D\)};
		\node (d5) [above=1.5cm of 5] {\(\D\)};
		
		\node[state,white,minimum size=1.4cm] (down) [below = -0.3cm of dots] {};
		\node[state,white,minimum size=1.4cm] (up) [above = -0.3cm of dots] {};
		
		\path 
		(2) edge	[left,ck]			node {\(\kappa\)}			(d2)
		(3) edge	[left,ck]			node {\(\kappa\)}			(d3)
		(4) edge	[left,ck]			node {\(\kappa\)}			(d4)
		(5) edge	[left,ck]			node {\(\kappa\)}			(d5)
		(2) edge	[cmu]				node {\(\mu_n\)}			(1)
		(3) edge	[bend left,cmu]		node {\(\mu_{n+1}\)}		(2)
		(2) edge	[bend left,cl]		node {\(\l_n\)}				(3)
		(3) edge	[bend left,cl]		node {\(\l_{n+1}\)}			(4)
		(4) edge	[bend left,cmu]		node {\(\mu_{n+2}\)}		(3)
		
		(down) edge	[out=180,in=315,cmu]		node{} 		(4)
		(4) edge	[in=180,out=45,cl]		node{} 				(up)
		(5) edge	[in=0,out=225,cmu]		node{} 				(down)
		(up) edge	[in=135,out=0,cl]		node{} 			(5)
		
		;
		
	\end{tikzpicture}
	\caption{The transition graph of \(X^\pn\) for finite \(N\) (see \eqref{qXn_rates} and proof of Lemma~\ref{hitproblemma}).\label{Fig:Xn}}
\end{figure}

\begin{lemma}\label{hitproblemma}
	\begin{equation*}
		\P_{n+1}(T_n^X< \infty) = \P_{n+1}(T_n^{X^{(n)}}< \infty), \quad n\in [N-1].
	\end{equation*}
\end{lemma}
\begin{proof}
	This is immediate from the fact that the discrete-time embeddings of the processes $X$ and $X^{(n)}$ have the same transition probabilities on the set $[n+1 : N]$.
\end{proof}
For $n \in [N-1]$, let the jump rates $q_W^{(n)}$ be defined as
\begin{equation*}
	\begin{split}
		q^{(n)}_{W}(i,i+1) & \coloneqq \l_{i+1}, \quad i\in [n:N-2], \\
		q^{(n)}_{W}(i,i-1) & \coloneqq \mu_{i+1}, \quad i\in [n+1:N-1],
	\end{split}
\end{equation*}
that is, $q^{(n)}_{W}$ arises from $q^{(n)}_{\cZ}$ by setting $\kappa =0$. Let $W^{(n)}$ be a process that follows the jump rates $q_{W}^{(n)}$, and let $T^{{W}^{(n)}}_n$ be the first time at which $W^{(n)}$ hits the state $n$. Also, let $W^{(n)}_{\text{eq}}$ be the random equilibrium state of $W^{(n)}$.
\begin{remark}
	\begin{itemize}
		\item[a)]
		Noting that --- except for the killing part in $q_X^{(n)}$ --- the jump rates $q^{(n)}_{X}$ and $q^{(n)}_{W}$ differ (as long as $X$ is in \([n+1:N]\)) only by a simple index shift, we immediately obtain the following Feynman--Kac representation of the hitting probabilities~of Lemma~\ref{hitproblemma}:
		\begin{equation}\label{FK}
			\P_{n+1}\left(T_n^{X^{(n)}}< \infty\right)= \E_n\left[\exp\left(-\kappa T_{n-1}^{{W}^{(n-1)}}\right)\right].
		\end{equation}
		\item[b)] From the detailed balance equation, we obtain
		\begin{equation}\label{detbal}
			\frac{\P(W^{(n-1)}_{\text{eq}}=n)}{\P(W^{(n-1)}_{\text{eq}}=n-1)} = \frac{\lambda_n}{\mu_{n+1}}. 
		\end{equation}
	\end{itemize}
\end{remark}
The remaining piece for completing the proof of Proposition~\ref{CorrProp} is given by the following lemma, which is a special case of Proposition~\ref{keyprop} stated in the next section. There we formulate and prove this proposition in a more general framework, since this may be of independent interest. 
\begin{lemma}\label{lemFKWZ}For \(n\in [2:N-1]\),
	\begin{equation*}
		\frac{\P(W^{(n-1)}_{\rm eq}=n)}{\P(W^{(n-1)}_{\rm eq}=n-1)}\E_n\left[\exp\left(-\kappa T_{n-1}^{{W}^{(n-1)}}\right)\right] = \frac{\P(\cZ^{(n-1)}_{\rm eq}=n)}{\P(\cZ^{(n-1)}_{\rm eq}=n-1)}.
	\end{equation*}
\end{lemma}
\begin{proof} 
	This follows by identifying the state space of the lemma with the state space $S$ of Proposition~\ref{keyprop}, and in particular the pair $(n-1, n)$ in the lemma with the pair $(\omega,\alpha)$ in the proposition. The role of $W^{(n)}$ is taken by $\left(\mathcal W, (\P^0)_{i\in S}^{}\right)$, and that of $\cZ^{(n)}$ by $\left(\mathcal W, (\P^\kappa)_{i\in S}^{}\right)$ in Proposition~\ref{keyprop}.
\end{proof}
\begin{proof}[Proof of Proposition~\ref{CorrProp}]
The assertion now follows by a straightforward combination of Lemmas~\ref{hitproblemma} and~\ref{lemFKWZ} with~\eqref{FK} and \eqref{detbal}.
\end{proof}
\section{Markov chains with reset and rebirth, and a biased detailed-balance relation}\label{sechlrs}
For a finite or countably infinite state space~$S$, consider jump rates $q^0(i,j)$ that belong to an irreducible positive recurrent continuous-time Markov chain on $S$, and let $\left(\mathcal W, (\P^0_i)_{i\in S}^{}\right)$ be its canonical model, with $\P_i^0(\mathcal W_0=i)=1$; a prototype example is the process $W^{(n)}$ defined after Lemma~\ref{hitproblemma}. Assume that~$S$ has two distinguished elements, which we call $\omega$ and $\alpha$ and for which
\begin{equation*}
	q^0_{}(\omega,i) = q^0_{}(i,\omega) =0 \mbox{ for } i \in S\setminus\{\omega, \alpha\}.
\end{equation*}
For $\kappa \ge 0$, let the jump rates $q^{\kappa}$ be given by
\begin{equation*}
	q^\kappa(i,\omega) =
	\begin{cases}
		q^0(\alpha,\omega)+ \kappa, \quad & i=\alpha,\\
		\kappa, & i \in S\setminus \{\omega,\alpha\},
	\end{cases}
	\qquad q^\kappa(i,j) = q^0_{}(i,j) \mbox{ for } j \neq \omega.
\end{equation*}
A stylised version of the transition graph  is depicted in Figure~\ref{Fig:qkappa}; the set $S\setminus \{\omega\}$ is encircled in red.

\begin{figure}[h]
	\begin{tikzpicture}[->,shorten >=1pt,auto,node distance=3cm,semithick]
		\tikzstyle{every state}=[draw=black,text=black]
		
		\node[state,minimum size=1cm] (1) {\(\omega\)};
		\node[state,minimum size=1cm] (2) [right=2.5cm of 1] {\(\alpha\)};
		\node[state,minimum size=1cm] (3) [right=1.5 of 2] {\(\eta\)};
		\node		(dots) [right=1cm of 3] {\(\cdots\)};
		\node[state,minimum size=1cm] (4) [right=1.5 of dots] {\(\xi\)};
		
		\draw[rounded corners=0.8cm, thick, red] ($(2.north west)+(-0.4,0.4)$) --
			($(4.north east)+(0.4,0.4)$) --
			($(4.south east)+(0.4,-0.4)$) --
			($(2.south west)+(-0.4,-0.4)$) --
			cycle;
		
		\path 	
		(2)		edge	[bend right=70,ck]		node {\(\kappa\)}				(1)
		(3)		edge	[bend right=70,ck]		node {\(\kappa\)}				(1)
		(4)		edge	[bend right=70,ck]		node {\(\kappa\)}				(1)
		(1)		edge	[]						node {}							(2)
		(2)		edge	[]						node {}							(1)
		;
	\end{tikzpicture}
	\caption{\label{Fig:qkappa} The transition graph corresponding to the jump rates \(q^\kappa\); arrows representing transitions between states inside the red bubble are omitted.}
\end{figure}

Let $\P^\kappa_i$ be the law induced by $\mathcal W$ when starting in $i$ and following the transition rates $q^\kappa$. In words, under the law $\P_i^{\kappa}$, the process $\mathcal W$ has, from any of its states, a jump rate $\kappa$ to $\omega$, in addition to the jump rate $q^0(\alpha,\omega)$ from $\alpha$ to $\omega$, which is present also under $\P_i^0$. As under $\P_i^0$, all the excursions of $\mathcal W$ from $\omega$ have to make their first step to $\alpha$; this is why we call $\left(\mathcal W, (\P^\kappa_i)_{i\in S}^{}\right)$ a chain with {\em (homogeneous) reset} (to state~$\omega$) and {\em rebirth} (via state $\alpha$).

Let $T_\alpha$ and $T_\omega$ be the first hitting times of the states $\alpha$ and $\omega$, respectively, and let $\pi^\kappa_{}(i)$, $i\in S$, denote the stationary probability weights for~$q^\kappa$. 
\begin{proposition}\label{keyprop}
	The ratios of the stationary probability weights in $\alpha$ and $\omega$ obey the following relationship (of Feynman--Kac type): 
	\begin{eqnarray*}
		\frac{\pi^0_{}(\alpha)}{\pi^0_{}(\omega)}\E_\alpha^0\left[\exp\left(-\kappa T_\omega\right)\right] = \frac{\pi^\kappa_{}(\alpha)}{\pi^\kappa_{}(\omega)}. 
	\end{eqnarray*}
\end{proposition}
\begin{proof}
	We will first show that, for all $i \in S$,
	\begin{equation} \label{eq:2}
		\pi^\kappa(i) = \frac{\mathbb{E}^\kappa_\omega \bigl[ \int^{R_\omega}_0 \one \{\mathcal W_t =i \}\, \dt \bigr]}{\mathbb{E}^{\kappa}_\omega (R_\omega)}\, ,
	\end{equation}
	where $R_\omega$ is the return time to $\omega$. The ergodic theorem tells us that
	\begin{equation}\label{ergodic}
		\frac{1}{t} \int^t_0 \one \{ \mathcal W^{}_t =i\} \dt \xrightarrow{t\to\infty} \pi^\kappa(i)\qquad \P_\omega^\kappa \mbox{ - a.s.}
	\end{equation}
	On the other hand, the renewal reward theorem says that the l.h.s. of~\eqref{ergodic} converges, $\P_\omega^\kappa$ - a.s. as $t\to \infty$, to the r.h.s. of~\eqref{eq:2}. Together with~\eqref{ergodic}, this proves~\eqref{eq:2}. 
	
	Since the holding times in $\omega$ have the same expectation under $\P_\omega^0$ and under $\P_\omega^\kappa$\,, we observe that
	\begin{equation*}
		\mathbb{E}^{0}_\omega \biggl[ \int^{R_\omega}_0 \one \{ \mathcal W^{}_t =\omega\} \,\dt \biggr] = \mathbb{E}^{\kappa}_\omega \biggl[ \int^{R_\omega}_0 \one \{\mathcal W^{}_t =\omega\}\, \dt \biggr].
	\end{equation*}
	This, combined with \eqref{eq:2} (which also holds for $\kappa = 0$), leads to
	\begin{equation} \label{eq:3}
		\frac{\frac{\pi^{\kappa}_{}(\alpha)}{\pi^{\kappa}_{}(\omega)}}{\frac{\pi^{0}_{}(\alpha)}{\pi^{0}_{}(\omega)}} = \frac{\mathbb{E}^{\kappa}_\omega \bigl[ \int^{R_\omega}_0 \one \{ \mathcal W^{}_t = \alpha\} \dt \bigr]}{\mathbb{E}^{0}_\omega \bigl[ \int^{R_\omega}_0 \one \{ \mathcal W^{}_t = \alpha\} \dt \bigr] }\, .
	\end{equation}
	Let us now consider a Poisson point process $\Pi$ with intensity $\kappa$ on $\mathbb R_+$, which is defined on a suitable enlargement of $\P^0_\omega$, under which it is independent of~$\mathcal W$. Let $T^\Pi$ be the first Poisson point after $T_\alpha$, so that $T:= T^\Pi -T_\alpha$ is Exp($\kappa$)-distributed and independent of $R_\omega$ both under $\P_\omega^0$ and under $\P_\alpha^0$. Thanks to the definition of $q^\kappa$, the return time $R_\omega$ under~$\P^\kappa_\omega$ has the same distribution as $T^\Pi\wedge R_\omega$ under~$\P^0_\omega$. Hence, with $\cS_t$ denoting the sojourn time of $\mathcal W$ in $\alpha$ until time $t$, the numerator of the RHS of \eqref{eq:3} becomes
	\begin{equation} \label{eq:5}
		\mathbb{E}^{\kappa}_\omega \biggl[ \int^{R_\omega}_0 \one \{\mathcal W^{}_t = {\alpha}\}\, \dt \biggr]= \mathbb{E}^{0}_\omega \left[\cS^{}_{T^\Pi\wedge R_\omega}\right] =\E_\alpha^0\left[\cS_{T\wedge R_\omega}\right] .
	\end{equation}
	Similarly, the denominator of the RHS of \eqref{eq:3} turns into
	\begin{equation} \label{eq:6}
		\mathbb{E}^{0}_\omega \biggl[ \int^{R_\omega}_0 \one \{ \mathcal W^{}_t =\alpha\} \dt \biggr] = \mathbb{E}^{0}_\omega [\cS^{}_{R_\omega}] = \mathbb{E}^{0}_\alpha [\cS^{}_{R_\omega}] .
	\end{equation}
	The above-stated properties of the random time $T$ give
	\begin{equation} \label{eq:7}
		\mathbb{E}^{0}_\alpha [e^{-\kappa R_\omega}] = \mathbb{P}^{0}_\alpha [R_\omega < T]\, .
	\end{equation}
	Thus equations \eqref{eq:3}--\eqref{eq:7} yield the following reformulation of Proposition~\ref{keyprop}:
	\begin{equation} \label{eq:8}
		\E^{0}_\alpha \bigl[\cS^{}_{T\wedge R^{}_\omega}\bigr] = \E^{0}_\alpha \bigl[\cS^{}_{R^{}_\omega}\bigr] \P^{0}_\alpha \left[R^{}_\omega <T\right].
	\end{equation}
	The l.h.s.\ of~\eqref{eq:8} can be decomposed as 
	\begin{align}
		\E^{0}_\alpha\left[\cS^{}_{T\wedge R_\omega}\right]\nonumber
		&=\E^{0}_\alpha\left[\cS^{}_{R^{}_\omega} \mid T>R^{}_\omega\right] \,\P^{0}_{\alpha}\left[T>R^{}_\omega\right] \nonumber\\
		&\hspace{1cm} + \E^{0}_\alpha\left[\cS^{}_{T}\mid T<R^{}_\omega\right] \, \P^{0}_{\alpha}[T<R^{}_\omega]\nonumber\\
		&= \E^{0}_\alpha\left[\cS^{}_{R^{}_\omega}\mid T>R^{}_\omega\right]\, \P^{0}_{\alpha}[T>R^{}_\omega] \nonumber\\
		&\hspace{1cm} +\E^{0}_\alpha\left[\cS^{}_{R^{}_\omega} - (\cS^{}_{R^{}_\omega} - \cS^{}_{T}) \mid T<R^{}_\omega\right]\, \P^{0}_{\alpha}[T<R^{}_\omega]\nonumber\\
		&= \E^{0}_\alpha\left[\cS^{}_{R^{}_\omega}\right]- \E^{0}_\alpha\left[\cS^{}_{R^{}_\omega} - \cS^{}_{T} \mid T<R^{}_\omega\right]\, \P^{0}_{\alpha}[T<R^{}_\omega]. 
		\label{eq:9}
	\end{align}
	On the event $\{T<R_\omega\}$, one has $\mathcal W_T \neq \omega$ \,$\mathbb{P}_\alpha^0$-a.s. Under the $q^0$-dynamics, the return to $\omega$ leads a.s.\ via the state $\alpha$, and collecting sojourn time in $\alpha$ after time $T$ can start only at the subsequent first hit of $\alpha$. Hence the strong Markov property gives
	\begin{equation*}\label{eq:10}
		\E^{0}_\alpha\bigl[ \cS^{}_{R^{}_\omega} - \cS^{}_{T} \mid T<R^{}_\omega\bigr] = \E^{0}_\alpha[\cS^{}_{R^{}_\omega}].
	\end{equation*}
	Plugging this into~\eqref{eq:9} immediately gives~\eqref{eq:8} and thus proves the proposition.
\end{proof}

\begin{remark}
	Note that, under $\mathbb{P}_i^0$, the stationary probability fluxes between $\omega$ and $\alpha$ follow the detailed-balance equation $\pi^0(\omega)q^0(\omega, \alpha) =\pi^0(\alpha)q^0(\alpha, \omega)$. Combining this with~Proposition~\ref{keyprop} tells us that
	\begin{equation*}
		\pi^\kappa(\alpha)q^0(\alpha,\omega) =\pi^\kappa(\omega)q^0( \omega,\alpha) \E_\alpha^0\bigl [\exp (-\kappa T_\omega ) \bigr],
	\end{equation*}
	which reflects the biasing of the stationary probability fluxes out of and into the state $\omega$, caused by the $\kappa$-transitions under $\mathbb{P}_i^\kappa$. Note also that, while the transitions within $S \setminus \{\omega \}$ are assumed irreducible but may otherwise be arbitrary, the result and its proof crucially rely on the constant resetting rate $\kappa$ from all states in $S \setminus \{\omega \}$ to $\omega$. \hfill
\end{remark}

\section{Flights and Siegmund duality}\label{Sec:siegmund}
Having established Part A of Theorem~\ref{thm_general}, we now lay the groundwork for proving Part B. The concept of \emph{Siegmund duality} allows us to interpret the desired relation as an instance of the first part, by introducing new processes that interchange the roles of \(X\) and \(\cZ\).

Let \( \mathcal{Z} = (\mathcal{Z}_t)_{t \ge 0} \) be a right-continuous, stochastically monotone Markov process with (possibly infinite) state space \( E \subseteq \mathbb{N}^\Delta \), where \(\mathbb{N}^\Delta\) is equipped with the usual order on \(\mathbb{N}\), extended by \(i < \Delta\) for all \(i \in \mathbb{N}\). If \( \Delta \in E \), we assume it is absorbing. In contrast to $Z$ under \eqref{cond_lambda}, $\mathcal{Z}$ is not safe against explosion; if it does explode, we send it to \( \Delta \) immediately after the first explosion time. For convenience, we adopt the convention that \( \Delta \in E \) only when \( \mathcal{Z} \) is explosive. As shown by Siegmund~\cite{Siegmund}, such a process admits a \emph{Siegmund dual} \( \mathcal{Z}^\star = (\mathcal{Z}_t^\star)_{t \ge 0} \), which is a Markov process on \(\Edelta \coloneqq E \cup \{\Delta\}\) satisfying the duality relation
\begin{equation}\label{sduality}
	\mathbb{P}(\mathcal{Z}_t \ge j \mid \mathcal{Z}_0 = i)
	= \mathbb{P}(i\geq \mathcal{Z}_t^\star \mid \mathcal{Z}_0^\star = j),
	\qquad i \in E,\; j \in \Edelta .
\end{equation}
Thus the semigroups of \( \mathcal{Z} \) and \( \mathcal{Z}^\star \) determine each other in a unique way. Moreover, the transition rates of \( \mathcal{Z}^\star \) can be expressed explicitly in terms of those of \( \mathcal{Z} \). The (rates of the) Siegmund dual may be calculated directly via Theorem~2 in \cite{Siegmund}, as also done in \cite{Enricos-thesis}; this is a straightforward but somewhat tedious task. We adopt an alternative approach here, which is based on the \emph{pathwise construction} of Siegmund duals as introduced by Clifford and Sudbury \cite{CS85}.

We begin by briefly recalling the building blocks of Clifford and Sudbury’s construction. These are the so-called \textit{flights}, which are non-decreasing functions \( f: E \to E \), that is, functions that preserve the order in \(E\). They will be the building blocks for the pathwise construction of the dual.

\subsection{Flight representations and duality}\label{ss:4.1}
The results of Clifford and Sudbury~\cite{CS85} provide a systematic method to construct, on a common probability space, versions of our process \( \mathcal{Z} \) starting from all possible states in \(E\), while almost surely preserving the order of their initial states. A central object in this construction is the set of flights (on \(E\)), that is,
\begin{equation*}
	\mathcal{F}_E \coloneqq \{\, f : E \to E \mid f \text{ is non-decreasing} \,\}.
\end{equation*}
Given a measure \(\gamma\) on \(\mathcal{F}_E\) satisfying
\begin{equation}\label{gammafin}
	\gamma^{(k)} \coloneqq \gamma\left(\{f \in \mathcal{F}_E \mid f(k) \neq k\}\right) < \infty \quad \text{for all } k \in E,
\end{equation}
let \(\Phi = \Phi_\gamma\) be a Poisson random measure on \([0,\infty) \times \mathcal{F}_E\) with intensity measure \({\rm d}t \times \gamma({\rm d}f)\). For any initial state \(\ell \in E\), the random flight configuration \(\Phi\) serves as a routing instruction for the path \(\mathcal{Z}_\Phi = \mathcal{Z}_{\Phi,\ell}\) defined as follows: When \(\mathcal{Z}_\Phi\) is in state \(k\) at time \(s\), find the point \((u,f)\) in the support of \(\Phi\) with the smallest \(u > s\) such that \(f(k) \neq k\); due to~\eqref{gammafin}, such a point exists almost surely. Then set \(\mathcal{Z}_\Phi(t) \coloneqq k\) for \(s \le t < u\), and \(\mathcal{Z}_\Phi(u) \coloneqq f(k)\). Iterating this procedure yields, for every \(\ell \in E\), a random sequence \(0 < T_1 < T_2 < \cdots\) of jump times. On the event \(\{\sup_{i}T_i < \infty\}\), we set \(\mathcal{Z}_\Phi(t) \coloneqq \Delta\) for all \(t \ge T_\infty \coloneqq \sup_i T_i\).

It is immediate from the construction that the resulting \(E\)-valued process \(\mathcal{Z}_\Phi = (\mathcal{Z}_{\Phi,\ell}(t))_{t \ge 0}\), for each \(\ell \in E\), is Markovian, stochastically monotone, and right-continuous. The jump rates of \(\mathcal{Z}_\Phi\) are given by
\begin{equation}\label{qij}
	q_{i,j} = \gamma\left(\{f \in \mathcal{F}_E \mid f(i) = j\} \right), \quad i, j \in E,\, i \neq j.
\end{equation}
We call \(\mathcal{Z}_\Phi\) the {\em process routed by} \(\Phi = \Phi_\gamma\).

A flight \(f\) can be visualised as a set of simultaneous arrows pointing from every \(j \in E\) to \(f(j)\), so that the process routed by \(\Phi_\gamma\) simply ``follows the arrows''; see~\cite{ECP} for an illustration.

Clifford and Sudbury \cite{CS85} show that, for any stochastically monotone and right-continuous Markov process \((\mathcal{Z}, \mathbb{P}_\ell)_{\ell \in E}\) that satisfies \(\mathcal{Z}_0 = \ell\) a.s.\ under \(\mathbb{P}_\ell\), there exists a {\em flight representation}, i.e.\ a Poisson random measure \(\Phi = \Phi_\gamma\) such that \(\mathcal{Z}_{\Phi,\ell}\) (constructed as above) has the same distribution as \(\mathcal{Z}\) under \(\mathbb{P}_\ell\) for all starting values \(\ell\). An explicit construction of the underlying Poisson random measure \(\Phi\) is given in~\cite{CS85}.

The flight representations of such Markov processes $\mathcal{Z}$ are not necessarily unique. In our examples, we will see that there is a natural and intuitive way to construct the corresponding Poisson measures of flights. 

Now, let us explain how to construct the Siegmund dual starting from the flight representation. The first ingredient is the notion of \emph{dual flight}: if \(f\) is a flight, its \emph{dual flight} is the map \(f^\star : \Edelta \to \Edelta\) defined by
\begin{equation*}\label{eq:dual_flight}
	f^\star(j) = \min \{\, i \in E : j \leq f(i) \,\},
\end{equation*}
where we adopt the convention that \(\min \varnothing = \Delta\), compare \cite[Fig.~3]{ECP} for an example of the graphical picture. In situations where \(\{ i \in E : j \leq f(i) \}\) is non-empty for all flights $f$ under consideration, the dual flights can be equivalently restricted to \(E\), and the inclusion of \(\Delta\) becomes unnecessary.

Let \(\Phi = \Phi_\gamma\) be a Poisson random measure satisfying condition~\eqref{gammafin}, and let \(\mathcal{Z}_\Phi\) be the process routed by \(\Phi\). We also fix a finite time horizon \(T > 0\) and denote by \(\Phi^{T}\) the restriction of \(\Phi\) to \([0, T] \times \mathcal{F}_E\). Finally, we define \(\Phi^{T,\star}\) as the Poisson measure on \([0, T] \times \mathcal{F}_{\Edelta}\) obtained as the pushforward of \(\Phi^{T}\) under the map \((s, g) \mapsto (T - s, g^\star)\). By construction, \(\Phi^{T,\star}\) has intensity \(\mathrm{d}t \times \gamma^\star(\mathrm{d}f^\star)\), where \(\gamma^\star\) is the pushforward of \(\gamma\) under the map \(f \mapsto f^\star\). Note that
\begin{equation*}
	f^\star(k) \neq k \;\Rightarrow\; f(i) \neq i \quad \text{for some } i \leq k,
\end{equation*}
and hence \(\gamma^\star\) satisfies condition~\eqref{gammafin}. Therefore, the process \(\mathcal{Z}_{\Phi^{T,\star}}\) routed by \(\Phi^{T,\star}\) is well defined. Moreover, it is not difficult to check that we almost surely have, for all \(t \in [0, T]\) and all \(n,m \in E\),
\begin{equation*}\label{pathwise-dual}
	\mathcal{Z}_{\Phi^{T},n}(T) \geq m
	\;\iff\;
	\mathcal{Z}_{\Phi^{T},n}(t) \geq \mathcal{Z}_{\Phi^{T,\star},m}((T - t)-)
	\;\iff\;
	n \geq \mathcal{Z}_{\Phi^{T,\star},m}(T-).
\end{equation*}
We refer to this relation as a \emph{pathwise Siegmund duality}. Note that fixing a finite time horizon is required only for this stronger notion of duality. If we are interested solely in the Siegmund duality~\eqref{sduality}, we may define \(\Phi^\star\) on \([0,\infty) \times \mathcal{F}_{\Edelta}\) as the Poisson measure with intensity \(\mathrm{d}t \times \gamma^\star(\mathrm{d}f^\star)\). In this case, the pathwise Siegmund duality implies the Siegmund duality \eqref{sduality} between the process \(\mathcal{Z}_\Phi\) and the process \(\mathcal{Z}_{\Phi^\star}\).

In the next section, we construct flight representations for a family of general birth–death processes with catastrophes.

\subsection{Flight representation for bdc processes}
Assume that we are given the following set of parameters: \(\kappa>0\), \(\ell_i \geq 0\) for $i>0$, and \(m_i\geq 0\) for \(i>1\). For $N \in \mathbb{N}_\infty$, let \(\mathcal{Z}^{N-1}\) be a bdc with jump rates \(q^{N-1}\) parametrised as follows: 
\begin{equation}\label{defq}\begin{split}
	\begin{aligned}
		q^{N-1}(i,i+1) & = \ell_i, &\quad & i\in [N-2],\\
		q^{N-1}(i,i-1) & = m_i + \kappa, &\quad &i\in [2:N-1], \\
		q^{N-1}(i,j) &= \kappa, &\quad &i\in [3:N-1], \, j\in [i-2].
	\end{aligned}
	\end{split}
\end{equation}
We take the state space of \(\mathcal{Z}^{N-1}\) to be \([N-1]^\Delta\), where the additional state \(\Delta\) is absorbing and can be reached only through explosion. Apart from this addition, $\mathcal{Z}^{N-1}$ is a bdc in the sense of \eqref{calZ_rates} up to a reparametrisation, which comes in handy as long as we do not consider paired processes.

We define the following flights on $\N_0^\Delta$:
\begin{itemize}
	\item For $i\in \mathbb N$, let $f_{i\uparrow}$ be the flight that sends $i$ to $i+1$ and leaves all $j\neq i$ unchanged.
	\item For $i \in [2:\infty]$, let $f_{i\downarrow}$ be the flight that sends $i$ to $i-1$ and leaves all $j\neq i$ unchanged.
	\item For $k \in \mathbb N$, let $f_{\downarrow k}$ be the flight that sends $j$ to $k$ for \emph{all} $j \in [k+1:\infty]$ and leaves all $j\le k$ and $j=\Delta$ unchanged. 
\end{itemize}
For $2\le N\le \infty$, we define the measure $\gamma_{N-1}$ as follows:
\begin{equation*}
	\gamma_{N-1} \coloneqq \sum_{i=1}^{N-2} \ell_i \delta_{f_{i\uparrow} |_{[N-1]^{\Delta}}^{}}+ \sum_{i=2}^{N-1} m_i\delta_{f_{i\downarrow} |_{[N-1]^{\Delta}}^{}} +\kappa \sum_{k=1}^{N-2}\delta_{f_{\downarrow k}|_{[N-1]^{\Delta}}^{}},
\end{equation*}
where $\vert$ denotes the restriction of a flight to a subset of its domain. Let $\Phi_{N-1}$ be a Poisson process with intensity measure ${\rm d}t\times \gamma_{N-1}({\rm d}f)$. It is readily checked that
\begin{align*}
	\gamma_{N-1}(\{f\in \mathcal F_{[N-1]^{\Delta}} \mid f(i)=i+1\}) &= \ell_i, &\quad& i\in [N-2], \\
	\gamma_{N-1}(\{f\in \mathcal F_{[N-1]^{\Delta}} \mid f(i)=i-1\})&= m_i+\kappa, &\quad& i\in [2:N-1], \\
	\gamma_{N-1}(\{f\in \mathcal F_{[N-1]^{\Delta}} \mid f(i)=j\})& = \kappa, &\quad& i\in [3:N-1], \, j\in [i-2]. 
\end{align*}
By ~\eqref{qij}, therefore, the process ${\color{black}{\mathcal{Z}_{\Phi_{N-1}}}}$ routed by $\Phi_{N-1}$ has the jump rates~\eqref{defq}, so $\Phi_{N-1}$ provides a flight representation of ${\color{black}{\mathcal{Z}^{N-1}}}$.

\subsection{Siegmund duality between bdc and bdk processes}
The dual flights of \(f_{i\uparrow}\), \(f_{i\downarrow}\), and \(f_{\downarrow k}\), as defined at the beginning of the previous subsection, are given by
\begin{equation*}
	f_{i\uparrow}^\ast = f_{(i+1)\downarrow}, \qquad 
	f_{i\downarrow}^\ast = f_{i\uparrow}, \qquad
	f_{\downarrow k}^\ast = f_{(k+1)\uparrow \Delta},
\end{equation*}
where
\begin{itemize}
	\item for \(k \in \mathbb{N}\), \(f_{k\uparrow \Delta}\) denotes the flight that sends
	\(j\) to \(\Delta\) for all \(j \in [k:\infty]\) and leaves all \(j < k\) and \(j = \Delta\) unchanged.
\end{itemize}

For restrictions to \( [N-1]^\Delta \), we obtain the analogous relations for \(i,k \in [N-2]\):
\begin{equation*}
	(f_{i\uparrow}\mid_{[N-1]^\Delta})^\ast = 
		f_{(i+1)\downarrow}\mid_{[N-1]^\Delta}, 
	\qquad
	(f_{i\downarrow}\mid_{[N-1]^\Delta})^\ast = 
		f_{i\uparrow}\mid_{[N-1]^\Delta}, 
	\qquad
	(f_{\downarrow k}\mid_{[N-1]^\Delta})^\ast = 
		f_{(k+1)\uparrow \Delta}\mid_{[N-1]^\Delta}.
\end{equation*}
These relations are complemented by
\begin{equation*}
	(f_{(N-1)\downarrow}\mid_{[N-1]^\Delta})^\ast
	= f_{(N-1)\uparrow \Delta}\mid_{[N-1]^\Delta},
\end{equation*}
which can be seen from
\begin{equation*}
	\left(f_{(N-1)\downarrow} \mid_{[N-1]^\Delta}\right)^\star(N-1)
		= \min\{\, i \in [N-1]^\Delta : N-1 \le f_{(N-1)\downarrow}(i) \,\}
		= \Delta.
\end{equation*}
Thus, the push-forward $\gamma_{N-1}^\star$ of $\gamma_{N-1}$ under the mapping $f\mapsto f^\star$ is
\begin{equation*}
	\begin{split}
		\gamma_{N-1}^\star &= \sum_{i=1}^{N-2} \ell_i \delta_{f_{(i+1)\downarrow} |_{[N-1]^\Delta}^{}}+ \sum_{i=2}^{N-2} m_i\delta_{f_{i\uparrow} |_{[N-1]^\Delta}^{}} + \sum_{k=1}^{N-2}(\kappa+\mathbbm{1}_{\{k=N-2\}}m_{N-1})\,\delta_{f_{(k+1)\uparrow \Delta}|_{[N-1]^\Delta}^{}}.
	\end{split}
\end{equation*}
As explained in Section~\ref{ss:4.1}, the measure $\gamma^\star_{N-1}$ encodes the jump rates of the Siegmund dual of $\mathcal{Z}^{N-1}$ via~\eqref{qij}, which yields the following result.

\begin{proposition}\label{Siegmund1}
For $N\in\N_\infty$, the bdc $\mathcal{Z}^{N-1}$ admits a Siegmund dual $\mathcal{Z}^{N-1,\star}$ on $[N-1]^\Delta$, which is a bdk with transition rates 
\begin{equation*}\begin{split}
	\begin{aligned}
		q^{{N-1},\star}(i,i+1) & = m_i, &\quad & i\in [2:N-2],\\
		q^{{N-1},\star}(i,i-1) & = \ell_{i-1} , &\quad &i\in [2:N-1], \\
		q^{{N-1},\star}(i,\Delta) &= (i-1)\kappa+\mathbbm{1}_{\{i=N-1\}}m_{N-1}, &\quad &i\in [2:N-1].
	\end{aligned}
	\end{split}
\end{equation*}
\end{proposition}

\section{Proof of Theorem~\ref{thm_general}~B}
\label{sec:proof_part_B}
Let us now return to our paired processes, the bdk $X= X^N$ and the bdc $Z=Z^{N-1}$, and get ready for the proof of Part~B of our main result.

\subsection{Siegmund dualities related to $X$ and $Z$}
The aim of this section is to prove the following result.
\begin{proposition}[Siegmund duality]\label{pSD}
	Under conditions \eqref{cond_lambda} and \eqref{cond_mu}, we have
	\begin{enumerate}
		\item The process \(\cZ\) admits a Siegmund dual \(\cZ^\star\coloneqq(\cZ^\star_t)_{t\geq 0}\) on \([N-1]^\D\), that is, for any \(i^\star\in [N-1]^\D,i\in [N-1]\), and \(t\geq 0\),
		\begin{equation*}
			\P(\cZ_t \geq i^\star \mid \cZ_0 = i) = \P(i \geq \cZ^\star_t \mid \cZ^\star_0 = i^\star).
		\end{equation*}
		The process \(\cZ^\star\) has absorbing states \(1\) and \(\D\) and transition rates 
		\begin{equation*}\label{eq:Zstar_rates}
			\begin{split}
				q^{}_{\cZ^\star}(i,i+1) & = (i-1)\mu_{i+1}, \quad i\in [N-2],\\
				q^{}_{\cZ^\star}(i,i-1) & = (i-1)\l_i, \quad i\in [2:N-1],\\
				q^{}_{\cZ^\star}(i,\D) &= (i-1)\kappa+(N-2)\mu_N \mathbbm{1}_{\{i=N-1\}}, \quad i\in [2:N-1].
			\end{split}
		\end{equation*}
		Furthermore, \(\cZ_t\) converges in distribution as \(t \to \infty\) to a random variable \(\cZ_{\mathrm{eq}}\), which is distributed according to the unique stationary distribution of \(\cZ\). In addition,
		\begin{align}\label{eq:a_absorption}
			\P(\cZ^\star \text{ absorbs in } 1 \mid \cZ^\star_0 = i) 
			= \ca_{i-1} 
			= \P(\cZ_{\mathrm{eq}} \ge i), 
			\qquad i \in [N-1],
		\end{align}
		and \(\lim_{i \to \infty} \ca_i = 0\) in the case $N=\infty$.
		\item Let $X^\circ=(X^\circ_t)_{t\geq 0}$ be the process on $[N]_0$ with transition rates 
		\begin{equation}\label{eq:Xcirc_rates}
			\begin{split}
				q^{}_{X^\circ}(i,i+1) & = (i+1)\mu_{i+1}, \quad i\in [0:N-1],\\
				q^{}_{X^\circ}(i,i-1) & = i\l_{i}+\kappa, \quad i\in [1:N],\\
				q^{}_{X^\circ}(i,j) & = \kappa, \quad i\in [2:N], \; j\in [i-2].
			\end{split}
		\end{equation}
		Then, the process \(X^\circ\) admits $X$ as Siegmund dual, that is, for any \(i^\circ\in[N]_0,i\in [N]_0^\Delta\), and \(t\geq 0\),
		\begin{equation*}
			\P(X_t^\circ \geq i\mid X_0^\circ=i^\circ)=\P(i^\circ \geq X_t \mid X_0=i).
		\end{equation*}
		In other words, $X=(X^\circ)^\star$. Furthermore, \(X^\circ_t\) converges in distribution, as \(t \to \infty\), to a random variable \(X^\circ_{\mathrm{eq}}\), which is distributed according to the unique stationary distribution of \(X^\circ\). In addition,
		\begin{equation*}
			\P(X^\circ_{\mathrm{eq}}\geq i) = b_{i}=\P(X \text{ absorbs in } 0 \mid X_0 = i),\quad i\in [N]_{0}, 
		\end{equation*}
		and \(\lim_{i \to \infty} b_i = 0\) in the case $N=\infty$.
	\end{enumerate}
\end{proposition} 
\begin{figure}[t]
	\centering
	\begin{tikzpicture}[->,shorten >=1pt,auto,node distance=3cm,semithick]
		\tikzstyle{every state}=[draw=black,text=black]
		
		\node[state,minimum size=1.2cm] (0X) {1};
		\node[state,minimum size=1.2cm] (1X) [right=1.5cm of 0X] {2};
		\node[state,minimum size=1.2cm] (2X) [right=1.5cm of 1X]{3};
		
		\node 			(dotsX) [right=1cm of 2X] {\(\cdots\)};
		\node[state,minimum size=1.2cm] (4X) [right=1cm of dotsX] {\(N\!-\!2\)};
		\node[state,minimum size=1.2cm] (5X) [right=1.5cm of 4X] {\(N\!-\!1\)};

		\node (d1X) [above=1.5cm of 1X] {\(\D\)};
		\node (d2X) [above=1.5cm of 2X] {\(\D\)};
		\node (d4X) [above=1.5cm of 4X] {\(\D\)};
		\node (d5X) [above=1.5cm of 5X] {\(\D\)};
		
		\node[state,minimum size=1.4cm,opacity=0] (downX) [below = -0.3cm of dotsX] {};
		\node[state,minimum size=1.4cm,opacity=0] (upX) [above = -0.3cm of dotsX] {};
		
		\path
		(1X) edge 		[cmu]				node {\(\l_2\)}					(0X)
		(1X) edge		[left,ck]			node {\(\kappa\)}				(d1X)
		(2X) edge		[left,ck]			node {\(2\kappa\)}				(d2X)
		(4X) edge		[left,ck]			node {\((N-3)\kappa\)}			(d4X)
		(5X) edge		[right,ck]			node {\((N\!-\!2)(\kappa\!+\!\mu_N)\)}				(d5X)
		(2X) edge		[bend left,cmu]		node {\(2\l_3\)}				(1X)
		(1X) edge		[bend left,cl]		node {\(\mu_3\)}				(2X)
		(4X) edge		[bend left,cl]		node {\((N-3)\mu_{N-1}\)}		(5X)
		(5X) edge		[bend left,cmu]		node {\((N-2)\l_{N-1}\)}		(4X)
		
		(downX) edge	[out=180,in=315,cmu]		node{} 		(2X)
		(2X) edge	[in=180,out=45,cl]		node{} 				(upX)
		(4X) edge	[in=0,out=225,cmu]		node{} 				(downX)
		(upX) edge	[in=135,out=0,cl]		node{} 				(4X)
		;

	
		\node[state,minimum size=1.2cm] (1) [below=3cm of 0X] {0};
		\node[state,minimum size=1.2cm] (2) [below=3cm of 1X] {1};
		\node[state,minimum size=1.2cm] (3) [below=3cm of 2X] {\(2\)};
		\node		(dots) [right=1cm of 3] {\(\cdots\)};
		\node[state,minimum size=1.2cm] (4) [below=3cm of 4X] {\(N\!-\!1\)};
		\node[state,minimum size=1.2cm] (5) [below=3cm of 5X] {\(N\)};
		
		\node[state,white,minimum size=1.4cm,opacity=0] (down) [below = -0.3cm of dots] {};
		\node[state,white,minimum size=1.4cm,opacity=0] (up) [above = 1.4cm of dots] {};
		\node[state,white,minimum size=1.4cm,opacity=0] (center) [right = 0.7cm of 3] {};
		
		\path 
		(1)		edge	[cl]					node {\(\mu_1\)}				(2)
		(2)		edge	[bend left,cmu]			node {\(\l_1\)}					(1)
		(2)		edge	[cl]					node {\(2\mu_2\)}				(3)
		(3)		edge	[bend left,cmu]			node {\(2\l_2\)}				(2)
		(4)		edge	[cl]					node {\(N\mu_N\)}				(5)
		(5)		edge	[bend left,cmu]			node {\(N\l_{N}\)}				(4)

		(2)		edge	[bend right=70,ck]		node {\(\kappa\)}				(1)
		(3)		edge	[bend right=70,ck]		node {\(\kappa\)}				(1)
		(3)		edge	[bend right=70,ck]		node {\(\kappa\)}				(2)
		(5) 	edge	[bend right=70,ck]		node {\(\kappa\)}				(4)
		
		(down)	edge	[cmu,in=315,out=180]	node {}							(3)
		(3)		edge	[cl]					node {}							(center)
		(up)	edge	[ck,out=180,in=90]		node {}							(3)
		(up)	edge	[ck,out=180,in=90]		node {}							(2)
		(up)	edge	[ck,out=180,in=70]		node {}							(1)
		
		(center)	edge	[cl,out=0,in=180]		node {}						(4)		
		(4)		edge	[ck,out=90,in=0]		node {}							(up)
		(5)		edge	[ck,out=90,in=0]		node {\(\kappa\)}				(up)
		(4)		edge	[cmu,out=225,in=0]		node {}							(down)
		;
	\end{tikzpicture}
	\caption{The transition graph of \(\cZ^\star\) (top) and \(X^\circ\) (bottom) for finite \(N\) (see Proposition~\ref{pSD}).\label{Fig:duals}}
\end{figure}
\begin{proof}
	The duality stated in part (1) is an immediate consequence of Proposition~\ref{Siegmund1} when we set \( \mathcal Z^{N-1} = Z \), corresponding to the choices \( \ell_i = i \lambda_{i+1} \) and \( m_i = (i-1)\mu_{i+1} \). For the duality in part (2), we apply Proposition~\ref{Siegmund1} to \( \mathcal{Z}^{N+1} = X^\circ + 1 \), which corresponds to the choices \( \ell_i = i\mu_i \) and \( m_i = (i-1)\lambda_{i-1} \), and recall that, for $N\in\N$, $m_{N+1}=N\lambda_N=0$. This yields \( (X^\circ + 1)^\star = X + 1 \), which, in turn, implies the desired duality.

	We now turn to the additional results concerning the processes \(\cZ\) and \(\cZ^\star\). First, since Condition~\ref{cond_mu} is satisfied, Lemma~\ref{lemma:conditions_mu_lambda} implies that \(\cZ\) absorbs in \(\{1,\Delta\}\) almost surely from any initial state. Thus, letting \(t \to \infty\) in the Siegmund duality relation yields, for any \(i^\star \in [N-1]^\Delta\) and \(i \in [N-1]\),
	\begin{equation}\label{sta-lim}
		\lim_{t\to\infty}\P(\cZ_t \geq i^\star \mid \cZ_0 = i) 
		= \P(\cZ^\star \text{ absorbs at } 1 \mid \cZ^\star_0 = i^\star)
		\eqqcolon \widehat{b}_{i^\star}.
	\end{equation}
	The map \(i^\star \mapsto \widehat{b}_{i^\star}\) is non-increasing, satisfies \(\widehat{b}_1 = 1\), and, under Condition~\ref{cond_lambda}, Lemma~\ref{lemma:conditions_mu_lambda} implies that \(\lim_{m\to\infty}\widehat{b}_m = 0\). Combining these properties of \((\widehat{b}_{i^\star})_{i^\star \in [N-1]^\Delta}\) with \eqref{sta-lim} shows that \(\cZ_t \to \cZ_{\mathrm{eq}}\) in distribution as \(t\to\infty\), where \(\cZ_{\mathrm{eq}}\) is a random variable whose distribution is independent of \(\cZ_0\) and characterised by
	\begin{equation*}
		\P(\cZ_{\mathrm{eq}} \geq i^\star) = \widehat{b}_{i^\star}.
	\end{equation*}
	As a limiting law, the law of $\cZ_{\mathrm{eq}}$ is stationary for the chain. Moreover, since $Z$ is irreducible, the existence of a stationary distribution implies non-explosion, which in turn yields uniqueness of the stationary distribution (see \cite[Thm.'s 3.5.2 and 3.5.3]{Norris}).

	The corresponding results for the processes \(X^\circ\) and \(X\) can be established in a completely analogous manner.
\end{proof}

As a consequence of Part~(1) of Proposition \ref{pSD}, we get the recursion \eqref{eq:rec_a_gen}:
\begin{corollary}\label{cor:rec_a}
	The vector \(\ca = (\ca_i)_{i\in[N]_0}\) of tail probabilities of \(X\) satisfies the recursion
	\begin{equation*}
		(\mu_{i+1}+\l_{i}+\kappa)\ca_i=\l_{i}\ca_{i-1}+\mu_{i+1}\ca_{i+1},\quad i \in [N-1],
	\end{equation*}
	complemented by the boundary conditions \(\ca_0=1\) and \(\ca_{N}=0\) when \(N<\infty\), and \/ \(\lim_{i\to\infty}\ca_i=0\) otherwise.
\end{corollary}
\begin{proof}
	The recursion follows immediately from a first-step decomposition of the absorption probabilities \eqref{eq:a_absorption} of \(\cZ^\star\). The boundary condition \(\ca_0=1\) is clear, as is \(\ca_N=0\) for finite \(N\). The condition \mbox{\(\lim_{i\to\infty}\ca_i=0\)} for \(N=\infty\) is also contained in Proposition~\ref{pSD}.
\end{proof}
We would like to note that Corollary~\ref{cor:rec_a} may also be obtained as a consequence of the stationarity conditions of \(\cZ\), but this is not obvious at first sight; see \cite[Proof of Prop.~6]{LKBW15} for the special case \eqref{special}.

As to Part~(1) of Proposition \ref{pSD}, we would like to emphasise that Siegmund duality is not an involution, so $X^\circ$ (with rates defined in \eqref{eq:Xcirc_rates}) is \emph{not} the Siegmund dual of $X$, but may (and will) be named the inverse Siegmund dual of $X$. This will become crucial when we now prove part~B of our main result.

\subsection{Completing the proof}\label{proofpartB}
Let us recapitulate where we stand at this point and what remains to be done. We have established all connections announced in Figure~\ref{Fig:intro} except for the lower horizontal one. That is, we can express ratios of the $b_i$ in terms of $a$ (and the birth-and-death rates) thanks to Part~A of Theorem~\ref{thm_general}, since the latter establishes the absorption probabilities of a bdk $X$ in terms of the tail probabilities of its paired bdc $Z$. Next, Proposition~\ref{pSD} has established $X$ as the Siegmund dual of the (wider-sense) bdc $X^\circ$ with tail probabilities $(b_i)$; as well as the bdk $Z^*$, with absorption probabilities $(a_i)$, as the Siegmund dual of $Z$. What is left to be done is to establish $a$ in terms of $b$. But although $a$ is now the vector of absorption probabilities of a bdk and $b$ the vector of tail probabilities of a bdc, we cannot invoke Part~A of Theorem~\ref{thm_general} directly, because $X^\circ$ is not an instance of $Z$ of \eqref{calZ_rates}; and because $Z^\star$ is not paired with $X^\circ$ in the sense of \eqref{X_rates} and \eqref{calZ_rates}. The first deficiency is easily mended: We work instead with $\overline{Z} \coloneqq X^\circ+1$; Figure~\ref{Fig:X_bar} (middle) shows its transition graph. We recognise in \(\overline Z\) an instance of our original process \(\cZ\), with state space \([N-1]\) replaced by \([N+1]\) and parameters \(\l_i,\mu_i\) replaced by \(\overline\l^{}_i\coloneqq \mu^{}_{i-1}\) for \(i\in[2:N+1]\) and \(\overline\mu^{}_i\coloneqq\l^{}_{i-2}\) for \(i\in[3:N+2]\), complemented by arbitrary \(\overline \l_1,\overline\mu_1,\overline\mu_2>0\) (and the rates still satisfying condition \eqref{calZ_rates}). As a consequence of Proposition~\ref{pSD}, \(\overline \cZ\) has a unique stationary distribution; we denote by \(\overline \ca \coloneqq (\overline \ca_i)_{i\in[N+1]_0}\) the corresponding tail probabilities. Let now \(\overline X\coloneqq (\overline X_t)_{t\geq 0}\) be the process paired with \(\overline \cZ\) in the sense of \eqref{X_rates} and \eqref{calZ_rates}; \(\overline X\) is defined on \([N+2]_0^\D\) and Figure~\ref{Fig:X_bar} (bottom) shows its transition graph. 

\begin{figure}[b]
	\centering
	\begin{tikzpicture}[->,shorten >=1pt,auto,node distance=3cm,semithick]
		\tikzstyle{every state}=[draw=black,text=black]
		
		\node[state,minimum size=1.2cm] (0)	{\(0\)};
		\node[state,minimum size=1.2cm] (1) [right=1.5cm of 0]{\(1\)};
		\node[state,minimum size=1.2cm] (2) [right=1.5cm of 1] {\(2\)};
		\node		(dots) [right=1cm of 2] {\(\cdots\)};
		\node[state,minimum size=1.2cm] (3) [right=1cm of dots] {\(N\)};
		\node[state,minimum size=1.2cm] (4) [right=1.5cm of 3] {\(N\!+\!1\)};
		\node[state,minimum size=1.2cm] (5) [right=1.5cm of 4] {\(N\!+\!2\)};
		
		\node (d1) [above=1.5cm of 1] {\(\D\)};
		\node (d2) [above=1.5cm of 2] {\(\D\)};
		\node (d3) [above=1.5cm of 3] {\(\D\)};
		\node (d4) [above=1.5cm of 4] {\(\D\)};
		\node (d5) [above=1.5cm of 5] {\(\D\)};
		
		\node[state,white,minimum size=1.4cm] (down) [below = -0.3cm of dots] {};
		\node[state,white,minimum size=1.4cm] (up) [above = -0.3cm of dots] {};
		
		\path
		(1) edge	[left,ck]			node {\small\(\kappa\)}						(d1)
		(2) edge	[left,ck]			node {\small\(2\kappa\)}					(d2)
		(3) edge	[left,ck]			node {\small\(N\kappa\)}					(d3)
		(4) edge	[left,ck]			node {\small\((N+1)\kappa\)}				(d4)
		(5) edge	[left,ck]			node {\small\((N+2)\kappa\)}				(d5)
		
		(1) edge	[cmu]				node {\small\(\overline\mu_1\)}				(0)
		(2) edge	[bend left,cmu]		node {\small\(2\overline \mu_2\)}			(1)
		(4) edge	[bend left,cmu]		node {\small\((N+1)\overline \mu_{N+1}\)}	(3)
		(5) edge	[bend left,cmu]		node {\small\((N+2)\overline \mu_{N+2}\)}	(4)
		
		(1) edge	[bend left,cl]		node {\small\(\overline \l_1\)}				(2)
		(3) edge	[bend left,cl]		node {\small\(N\overline \l_{N}\)}			(4)
		(4) edge	[bend left,cl]		node {\small\((N+1)\overline\l_{N+1}\)}		(5)
		
		(down) edge	[out=180,in=315,cmu]	node{} 				(2)
		(2) edge	[in=180,out=45,cl]		node{} 				(up)
		(3) edge	[in=0,out=225,cmu]		node{} 				(down)
		(up) edge	[in=135,out=0,cl]		node{} 				(3)
		;
		
		\node[state,minimum size=1.2cm] (b1) [above=2.5cm of 1]{\(3\)};
		\node[state,minimum size=1.2cm] (b2) [right=1.5cm of b1] {\(4\)};
		\node		(bdots) [right=1cm of b2] {\(\cdots\)};
		\node[state,minimum size=1.2cm] (b3) [right=1cm of bdots] {\(N\)};
		\node[state,minimum size=1.2cm] (b4) [right=1.5cm of b3] {\(N\!+\!1\)};
		
		\node (bd2) [above=1.5cm of b2] {\(\D\)};
		\node (bd3) [above=1.5cm of b3] {\(\D\)};
		\node (bd4) [above=1.5cm of b4] {\(\D\)};
		
		\node[state,white,minimum size=1.4cm] (bdown) [below = -0.3cm of bdots] {};
		\node[state,white,minimum size=1.4cm] (bup) [above = -0.3cm of bdots] {};
		
		\path
		(b2)	edge	[left,ck]			node {\small\(\kappa\)}								(bd2)
		(b3)	edge	[left,ck]			node {\small\((N-3)\kappa\)}						(bd3)
		(b4)	edge	[right,ck]			node {\small\((N-2)(\kappa+\overline\l_{N+1})\)}	(bd4)
		(b2)	edge	[cmu]				node {\small\(\overline \mu_4\)}					(b1)
		(b4)	edge	[bend left,cmu]		node {\small\((N-2)\overline \mu^{}_{N+1}\)}		(b3)
		(b3)	edge	[bend left,cl]		node {\small\((N-3)\overline \l^{}_{N}\)}			(b4)
		(bdown)	edge	[out=180,in=315,cmu]	node{}											(b2)
		(b2)	edge	[in=180,out=45,cl]		node{}											(bup)
		(b3)	edge	[in=0,out=225,cmu]		node{}											(bdown)
		(bup)	edge	[in=135,out=0,cl]		node{}											(b3)
		;

		\node[state,minimum size=1.2cm] (c1) [below=3cm of 1] {\(1\)};
		\node[state,minimum size=1.2cm] (c2) [below=3cm of 2] {\(2\)};
		\node		(cdots) [right=1cm of c2] {\(\cdots\)};
		\node[state,minimum size=1.2cm] (c3) [below=3cm of 3] {\(N\)};
		\node[state,minimum size=1.2cm] (c4) [below=3cm of 4] {\(N\!+\!1\)};
		
	 	\node[state,white,minimum size=1.4cm,opacity=0] (c_down) [below = -0.3cm of cdots] {};
		\node[state,white,minimum size=1.4cm,opacity=0] (c_up) [above = 1.4cm of cdots] {};
		\node[state,white,minimum size=1.4cm,opacity=0] (c_center) [right = 0.7cm of c2] {};
		
		\path 
		(c1)		edge	[cl]					node {\(\overline \l_2\)}			(c2)
		(c2)		edge	[bend left,cmu]			node {\(\overline \mu_3\)}			(c1)
		(c3)		edge	[cl]					node {\(N\overline \l_{N+1}\)}		(c4)
		(c4)		edge	[bend left,cmu]			node {\(N\overline \mu_{N+2}\)}		(c3)
		(c2)		edge	[bend right=70,ck]		node {\(\kappa\)}					(c1)
		(c4)		edge	[bend right=70,ck]		node {\(\kappa\)}					(c3)
		(c_down)	edge	[cmu,in=315,out=180]	node {}								(c2)
		(c2)		edge	[cl]					node {}								(c_center)
		(c_up)		edge	[ck,out=180,in=90]		node {}								(c2)
		(c_up)		edge	[ck,out=180,in=70]		node {}								(c1)
		(c_center)	edge	[cl,out=0,in=180]		node {}								(c3)
		(c3)		edge	[ck,out=90,in=0]		node {}								(c_up)
		(c4)		edge	[ck,out=90,in=0]		node {}								(c_up)
		(c3)		edge	[cmu,out=225,in=0]		node {}								(c_down)
		;
		
	\end{tikzpicture}
	\caption{From top to bottom: the transition graphs of \(\vec\cZ ^{\star}_{}\) (see proof of Lemma~\ref{lemma:barX}), \(\overline X\), and \(\overline \cZ\) (see beginning of Sec.~\ref{proofpartB}) for finite \(N\).\label{Fig:X_bar}}
\end{figure}

As for \(X\), states \(0\) and \(\D\) are absorbing for \(\overline X\); let \(\overline b\coloneqq (\overline b_i)_{i\in[N+2]_0}\) be the vector of absorption probabilities in \(0\) when starting from \(i\in[N+2]_0\). Thanks to Lemma~\ref{lemma:conditions_mu_lambda}, we know that assumption \eqref{cond_mu} implies that \(\P(\overline X \text{ absorbs in } \left\{0,\D\right\})=1\), while \eqref{cond_lambda} implies that \(\lim_{i\to \infty} \overline b_i =0\). Part~A of Theorem~\ref{thm_general} then gives us \(\overline b\) as a function of \(\overline \ca\), namely
\begin{equation}\label{eq:barb_bara}
	\frac{\overline b_i}{\overline b_2} = \frac{\overline \ca_{i-2}-\overline \ca_{i-1}}{\overline \ca_0-\overline \ca_1}\prod_{j=1}^{i-2}\frac{\overline \mu_{j+2}}{\overline \lambda_{j+1}}, \quad i\in [2:N+2]. 
\end{equation} 

By definition of \(\overline Z\), we have \(\overline{a}_i = \P(\overline Z >i) = \P(X^\circ \geq i)\); and thanks to Proposition~\ref{pSD}, we know that $\P(X^\circ \geq i)=b_i$, so
\begin{equation}\label{eq:bar_a}
	\overline \ca_i = b_i,\quad i\in[N]_0.
\end{equation}
In the light of \eqref{eq:barb_bara} and \eqref{eq:bar_a}, our task of determining $a$ as a function of $b$ therefore boils down to finding $a$ as a function of $\overline b$. This, however, is not immediate because of the second deficit: $Z^\star$ is not paired with $X^\circ$ or with $X^\circ+1=\overline Z$, or, put differently, \(\overline X\), the process paired with $\overline Z$, does \emph{not} equal \(\cZ^\star\). So Proposition~\ref{pSD} cannot be applied immediately. However, we will see that the process $\vec{Z}^\star\coloneqq Z^\star+2$, on its state space (which is a subset of the state space of $\overline{X}$), has the same hitting probabilities as $\overline X$. The following lemma provides us with this missing piece of the puzzle.
\begin{lemma}\label{lemma:barX}
	The vectors \(\ca\) and \(\overline b\) of absorption probabilties of $Z^\star$ and absorption probabilities of \(\overline X\), respectively, are connected via
	\begin{equation*}\label{eq:a_bbar}
		\ca^{}_{i} =\frac{\overline b^{}_{i+3}}{\overline b^{}_3},\quad i\in[N-2]_0.
	\end{equation*}
\end{lemma}
\begin{proof} 
	Consider the continuous-time Markov chain \(\vec {\cZ} ^\star \coloneqq \cZ^*+2\) on $[3:N+1]^\Delta$; 
	see Figure~\ref{Fig:X_bar} (top) for its transition graph. Clearly Proposition~\ref{pSD} grants us that
	\begin{equation*}
		\ca_i = \P(\vec \cZ^\star \textit{ absorbs in } 3 \mid \vec\cZ^\star_0 = i+3)
	\end{equation*}
	for \(i \in [N-2]_0\). So what we have to do is to express the absorption probabilities of $\vec{Z}^*$ in 3 in terms of the absorption probabilities of \(\overline X\) in \(0\). Let us, therefore, consider the embedded discrete-time Markov chains of \(\vec \cZ^\star_{}\) and \(\overline X\). For $i \in [4:N]$, the probabilities for the transitions $i \to i-1, i \to i+1$, and $i \to \D$ are the same for both embedded chains. The same holds for the transition to $N$ when starting in $N+1$. Again when starting in $N+1$, the probability of $\vec \cZ^\star$ to make its first step to $\Delta$ agrees with the probability of $\overline X$ to make its first step to $\{\Delta, N+2\}$, which, in turn, implies absorption in $\Delta$ due to $\overline{\mu}_{N+2}=0$. Altogether, therefore, we have for \(i\in[3:N+1]\) that
	\begin{equation}\label{eq:hitting_times}
		a_{i-3}= \P(T_3^{\vec \cZ^\star}<T_\D^{\vec \cZ^\star}\mid \vec \cZ^\star_0 = i) = \P(T_3^{\overline X}<T_\D^{\overline X}\mid \overline X_0 = i),
	\end{equation}
	where $T_k^{\vec \cZ^\star}$ and $T_k^{\overline X}$ denote the hitting times of $\vec \cZ^\star$ and $\overline X$ in $k$, respectively. Using the Markov property, we obtain
	\begin{equation*}
		\overline b_{i}=\P \big (T_0^{\overline X}<T_\D^{\overline X}\mid \overline X_0=i \big )=\P \big (T_{3}^{\overline X}<T_\D^{\overline X}\mid \overline X_0=i \big ) \ts \P \big (T_0^{\overline X}<T_\D^{\overline X}\mid \overline X_0=3 \big ).
	\end{equation*}
	Hence,
	\begin{equation*}
		\P\big(T_{3}^{\overline X}<T_\D^{\overline X}\mid \overline X_0=i\big)=\frac{\overline b_{i}}{\overline b_{3}}.
	\end{equation*} 
	Combining this with \eqref{eq:hitting_times} gives the conclusion.
\end{proof}

To complete the proof of Theorem~\ref{thm_general}~B, we now divide~\eqref{eq:barb_bara} by the same equation for \(i=3\) and obtain
\begin{equation*}
	\frac{\overline b^{}_i}{\overline b^{}_3} = \frac{\overline \ca^{}_{i-2}- \overline \ca^{}_{i-1}}{\overline \ca^{}_{1}-\overline \ca^{}_{2}}\prod_{j=2}^{i-2}\frac{\overline \mu^{}_{j+2}}{\overline \l^{}_{j+1}}, \quad i\in[2:N+2].
\end{equation*}
Inserting this into Lemma~\ref{lemma:barX} and combining with~\eqref{eq:bar_a} and the definition of the \(\overline\l^{}_i\) and~\(\overline\mu^{}_i\) leads to the conclusion. \hfill \qed


\section{An application}\label{sec:app}
As announced in the introduction, we now briefly review the Moran model for a finite population with two types under selection and mutation, and its diffusion limit. We then use our results to link the absorption probabilities of the (line counting process of) the killed ancestral selection graph (kASG) to the stationary distribution of the (line counting process of the) pruned lookdown ancestral selection graph (pLD-ASG). We will first present the result in the diffusion limit, where the application of Theorem~\ref{thm_general} is straightforward, and then move on to the finite case, which requires further effort.

\subsection{The two-type Moran model with selection and mutation and its diffusion limit}\label{sec:Moran}
Consider a population of constant finite size \(N>0\) consisting of two types of individuals, type \(0\) and type \(1\), which evolves as follows (see \cite{JSP} for a review). An individual of type \(i\in\{0,1\}\) may, at any time, do either of two things: it may reproduce, at rate \(1\) for type-\(1\) individuals and at rate \(1+s^N\), \(s^N>0\), for type-\(0\) individuals; or it may mutate, at rate \(u^N > 0\). (From now on, we make the dependence on \(N\) explicit.) When an individual reproduces, its single offspring inherits the parent's type and replaces a uniformly chosen individual, possibly its own parent. When an individual mutates, the new type is \(j\in\{0,1\}\) with probability \(\nu_j\in(0,1)\); \(\nu_0 + \nu_1 = 1\). The evolution of the type composition in the population is fully described by the process \(Y^N\coloneqq (Y_t^{N})_{t\geq 0}\), where \(Y_t^N\) denotes the number of type-\(1\) individuals in the population at time \(t\). The process \(Y^N\) is a continuous-time birth-death process on \([N]_0\) with generator matrix \((q^{}_{Y^N}(i,j))_{i,j \in [N]_0}\) defined by the nontrivial transition rates 
\begin{align*}
	q^{}_{Y^N}(i,i+1) &=i\myfrac{N-i}{N} + u^N\nu_1(N-i),\\
	q^{}_{Y^N}(i,i-1) &=(1+s^N)i\myfrac{N-i}{N} + u^N\nu_0 i.
\end{align*}
We denote by \(\pi^N\deq (\pi^N_i)_{i\in[N]_0}\) the (reversible) stationary distribution of \(Y^N\), which is given by
\begin{equation*}\label{stationarydistributionfinite}
	\pi^{N}_i \deq C^N \prod_{j=1}^{i-1} \frac{q^{}_{Y^{N}}(j,j+1)}{q^{}_{Y^{N}}(j+1,j)},
\end{equation*}
where the empty product is \(1\) and \(C^N\) is a normalising constant chosen so that \(\sum_{i=0}^{N}\pi_i^N = 1\). We denote by \(Y^N_{\text{eq}}\) a random variable with distribution \(\pi^N\).
\smallskip

In the diffusion limit, the Moran model turns into the Wright--Fisher diffusion $\cY$ mentioned in the introduction. Based on the notion of the ancestral selection graph introduced by Krone and Neuhauser \cite{KN97}, the kASG and the pLD-ASG were originally formulated in this limit, but have also been established for the (finite-)$N$ Moran model, see \cite{BCH18,JSP}. Here, too, the kASG is a genealogical tool to investigate the type distribution of a population at equilibrium, whereas the pLD-ASG does the same for the type distribution of the \emph{ancestors} of the current population.

\smallskip
The kASG for a finite population of \(N\) individuals\footnote{Strictly speaking, what we define here is the \emph{line-counting process} of the kASG, but since we do not need the full graphical structure in this work, we simply speak of the kASG.} is a special case of the process \(X=X^N\) introduced in Section~\ref{Sec:results} with
\begin{equation}\label{def_kASGfinite}
	\l_i^{R,N} \deq s^N \myfrac{N-i}{N}, \quad \mu_i^{R,N} \deq \myfrac{i-1}{N} + u^{N} \nu_1, \quad \text{and } \, \kappa^{R,N} \deq u^{N} \nu_0
\end{equation}
for \(i\in[N]\). We denote this process by \(R^N \deq (R^N_t)_{t \geq 0}\); its transition graph is shown in Figure~\ref{Fig:R_finite}.
\begin{figure}[t]
	\centering
	\begin{tikzpicture}[->,shorten >=1pt,auto,node distance=3cm,semithick]
		\tikzstyle{every state}=[draw=black,text=black]
		
		\node[state,minimum size=1.2cm] (0) {0};
		\node[state,minimum size=1.2cm] (1) [right of=0] {1};
		\node[state,minimum size=1.2cm] (2) [right of=1]{2};
		\node		(m_dots) [right=0.6cm of 2] {\(\cdots\)};
		\node[state,minimum size=1.2cm] (3) [right=0.6cm of m_dots] {\(i\)};
		\node[state,minimum size=1.2cm] (4) [right of=3] {\(i+1\)};
		\node		(r_dots) [right=0.6cm of 4] {\(\cdots\)};
		\node[state,minimum size=1.2cm] (5) [right=0.6cm of r_dots] {\(N\)};
		
		\node (d1) [above=1.5cm of 1] {\(\D\)};
		\node (d2) [above=1.5cm of 2] {\(\D\)};
		\node (d3) [above=1.5cm of 3] {\(\D\)};
		\node (d4) [above=1.5cm of 4] {\(\D\)};
		\node (d5) [above=1.5cm of 5] {\(\D\)};
		
		\node[state,white,minimum size=1.2cm] (m_down) [below = -0.3cm of m_dots] {};
		\node[state,white,minimum size=1.2cm] (m_up) [above = -0.3cm of m_dots] {};
		\node[state,white,minimum size=1.2cm] (r_down) [below = -0.3cm of r_dots] {};
		\node[state,white,minimum size=1.2cm] (r_up) [above = -0.3cm of r_dots] {};
		
		\path
		(1) edge	[cmu]					node {\(u^N\nu_1\)}						(0)
		(1) edge	[ck, right]				node {\(u^N\nu_0\)}						(d1)
		(3) edge	[ck, right]				node {\(iu^N\nu_0\)}					(d3)
		(4) edge	[ck, right]				node {\((i+1)u^N\nu_0\)}				(d4)
		(2) edge	[ck, right]				node {\(2u^N\nu_0\)}					(d2)
		(5) edge	[ck, right]				node {\(Nu^N\nu_0\)}					(d5)
		(2) edge	[cmu, bend left]		node {\(2(\frac{1}{N}+u^N\nu_1)\)}		(1)
		(1) edge	[cl, bend left,above]	node {\(\frac{N-1}{N}s^N\)}				(2)
		(3) edge	[cl, bend left,above]	node {\(i\frac{N-i}{N}s^N\)}			(4)
		(4) edge	[cmu, bend left]		node {\((i+1)(\frac{i}{N}+u^N\nu_1)\)}	(3)

		(m_down) edge	[out=180,in=315,cmu]			node{} 				(2)
		(2) edge		[in=180,out=45,cl]				node{} 				(m_up)
		(3) edge		[in=0,out=225,cmu]				node{} 				(m_down)
		(m_up) edge		[in=135,out=0,cl]				node{} 				(3)
		(r_down) edge	[out=180,in=315,cmu]			node{} 				(4)
		(4) edge		[in=180,out=45,cl]				node{} 				(r_up)
		(5) edge		[in=0,out=225,cmu]				node{} 				(r_down)
		(r_up) edge		[in=135,out=0,cl]				node{} 				(5)
		;
		
	\end{tikzpicture}

	\begin{tikzpicture}[->,shorten >=1pt,auto,node distance=3cm,semithick,state/.style={circle, draw, minimum size=1cm}]
		\tikzstyle{every state}=[draw=black,text=black]
		
		\node[state,minimum size=1.2cm] (1) {1};
		\node		(m_dots) [right=1cm of 1]{\(\cdots\)};
		\node[state,minimum size=1.2cm] (2) [right=1cm of m_dots] {\(i-1\)};
		\node[state,minimum size=1.2cm] (3) [right=2.5cm of 2] {\(i\)};
		\node[state,minimum size=1.2cm] (4) [right=2.5cm of 3] {\(i+1\)};
		\node		(r_dots) [right=1cm of 4] {\(\cdots\)};
		\node[state,minimum size=1.1cm] (5) [right=1cm of r_dots] {\(N\)};
		
		\node[state,white,minimum size=1.4cm] (m_down) [below = -0.3cm of m_dots] {};
		\node[state,white,minimum size=1.4cm] (m_up) [above = 1.4cm of m_dots] {};
		\node[state,white,minimum size=1.4cm] (r_down) [below = -0.3cm of r_dots] {};
		\node[state,white,minimum size=1.4cm] (r_up) [above = 1.4cm of r_dots] {};
		
		\path 
		(2) edge	[above,cl]			node {\((i-1)\frac{N-(i-1)}{N}s^N\)}		(3)
		(3) edge	[bend left,cmu]		node {\((i-1)(\frac{i}{N}+ u^N\nu_1)\)}		(2)
		(3) edge	[above,cl]			node {\(i\frac{N-i}{N}s^N\)}				(4)
		(4) edge	[bend left,cmu]		node {\(i(\frac{i+1}{N}+ u^N\nu_1)\)}		(3)

		(3) edge	[bend right=70,ck]	node {\(u^N\nu_0\)}			(2)
		(4) edge	[bend right=70,ck]	node {\(u^N\nu_0\)}			(2)
		(4) edge	[bend right=70,ck]	node {\(u^N\nu_0\)}			(3)
		
		(4) edge	[ck,out=90,in=0]	node {}					(2.6,2.3)
		(3) edge	[ck,out=90,in=0]	node {}					(2.6,2.3)
		(2) edge	[ck,out=90,in=0]	node {}					(2.6,2.3)
		
		(r_up) edge		[ck,out=180,in=90]		node {}				(4)
		(r_up) edge		[ck,out=180,in=90]		node {}				(3)
		(r_up) edge		[ck,out=180,in=90]		node {}				(2)
		(m_up) edge		[ck,out=180,in=90]		node {}				(1)
		(5) edge		[ck,out=90,in=0]		node {}				(r_up)
		
		(1) edge		[cl,out=0,in=180]		node {}				(1.3,0)
		(4) edge		[cl,out=0,in=180]		node {}				(12.6,0)
		(2.7,0) edge	[cl,out=0,in=180]		node {}				(2)
		(14,0) edge		[cl,out=0,in=180]		node {}				(5)
		
		(m_down) edge	[cmu,out=180,in=315]	node {}				(1)
		(r_down) edge	[cmu,out=180,in=315]	node {}				(4)
		(2) edge		[cmu,out=225,in=0]		node {}				(m_down)
		(5) edge		[cmu,out=225,in=0]		node {}				(r_down)
		
		;
	\end{tikzpicture}
	\caption{The transition graphs of \(R^N\) (top; see \eqref{def_kASGfinite}) and $L^N$ (bottom; see \eqref{def:pLD-ASG_finite}). \label{Fig:R_finite}}
\end{figure}
In line with the notation introduced in Section~\ref{Sec:results}, we denote by \(b^N\coloneqq (b^N_i)_{i\in [N]_0}\) the absorption probabilities in \(0\) of \(R^N\) when starting from state \(i\). The \((b^N_i)_{i\in [N]_0}\) are linked to \(Y^N\) through the following relationship, which is a special case of \cite[Corollary 2.4]{luigi}:
\begin{equation}\label{cor_bY_finite_eq}
	b^N_i = \P(R^N \text{ absorbs in } 0 \mid R^N_0=i) = \E \Big [ \frac{(Y^{N}_{\text{eq}})^ {\underline i}}{N^{\underline i}}\Big ], \quad i\in [N]_0,
\end{equation}
where \(\ell^{\underline{i}}\deq \frac{\ell!}{(\ell-i)!}\) is the falling factorial for an integer \(\ell\). In words, \(b^N_i\) is the probability to obtain type-1 individuals only when sampling \(i\) times without replacement from the stationary population. In this sense, \eqref{eq:rec_b_gen} now has the meaning of a \emph{sampling recursion}; explicitly, it reads
\begin{equation} \label{rec_b_finite}
	\Big (\frac{i-1}{N}+ s^N\frac{N-i}{N} + u^N \Big )b^N_i = \Big (\frac{i-1}{N}+u^N\nu_1\Big )b^N_{i-1} + s^N \frac{N-i}{N}b^N_{i+1}, \quad i \in [N],
\end{equation}
together with the boundary condition \(b^N_0=1\) and the convention \(b^N_{N+1}=0\). 

\smallskip 
In contrast, the pLD-ASG is designed to study the type distribution of the present sample’s ancestor in the distant past, see \cite{C17} for $N<\infty$ and \cite{LKBW15} for the diffusion limit. In the case of a population of \(N\) individuals the pLD-ASG\footnote{As for the kASG, we speak here of the pLD-ASG instead of its line-counting process.}, which we denote by $L^N$, has the transition rates in \cite[eq. (4.7)]{C17} and the transition graph in Figure~\ref{Fig:R_finite}. Thus, $L^N$ is a special case of the bdc \(Z=\cZ^{{N}}\) whose dynamics is specified in~\eqref{calZ_rates}, but note that the state space is now $[N]$ rather than $[N-1]$. The parameters read
\begin{equation}\label{def:pLD-ASG_finite}
	\l_i^{L,N} \deq s^N \myfrac{N-(i-1)}{N}, \quad \mu_i^{L,N} \deq \myfrac{i-1}{N} + u^N \nu_1, \quad \text{and } \, \kappa^{L,N} \deq u^N \nu_0
\end{equation}
for \(i\in[N]\). \normalcolor

In line with the notation introduced for \(\cZ\) of Section~\ref{Sec:results}, we denote by \(w^N\deq(w_i^N)_{i\in[N]}\) the stationary distribution of \(L^N\), and by \(a^N\deq(a_i^N)_{i\in[N]_0}\) the vector of the corresponding tail probabilities, so \(a^N_i = \P(L^N_{\text{eq}}>i) \), where \(L^N_{\text{eq}}\) is a random variable with distribution \(w^N\).
The recursion \eqref{eq:rec_a_gen} for the tail probabilities now reads explicitly
\begin{equation}\label{rec_a_finite}
	\Big(\frac{i+1}{N}+s^N \frac{N-i}{N}+u^N\Big)a_i^N = \Big(\frac{i+1}{N}+u^N \nu_1\Big)a_{i+1}^N + s\frac{N-i}{N}a_{i-1}^N,
\end{equation}
together with the boundary conditions \(a_0^N=1\) and \(a^N_N=0\). 
\smallskip

Let us briefly hint at how the pLD-ASG is connected to the type composition of the ancestral population in the distant past. Let \(g^{N}(i)\), \(i\in[N]_0\), be the probability that the individual at present whose progeny will take over in the population at some later time is unfit, given that there are \(i\) unfit individuals at present. Alternatively, due to time homogeneity, \(g^{N}(i)\) is the probability that the population at present descends from an unfit individual in the distant past, given that, in the past, there were \(i\) unfit individuals. It was proved in \cite{C17} that
\begin{equation}\label{eq:g(y)}
	g^{N}(i)=1-(N-i)\sum_{j=1}^N a^N_j\frac{i^{\underline j}}{N^{\underline {j-1}}},\quad i\in[N]_0.
\end{equation}

Let us now recall the \emph{weak selection-weak mutation regime}~\eqref{special}~of the finite Moran model, that is, we assume that
\begin{equation}\label{wswm}
	\lim_{N\to\infty} N s^{N}=\s\geq 0\quad\text{and}\quad\lim_{N\to\infty}Nu^{N}=\t>0
\end{equation}
(compare \eqref{special}).
In this setting, it is well known that, if \(Y_0^{N}/N\to y\in[0,1]\) as \(N\to \infty\), the rescaled process \((Y^{N}_{Nt}/N)^{}_{t\geq 0}\) converges in distribution, as \(N\to\infty\), to the \emph{Wright--Fisher diffusion with selection and mutation}, that is, to the solution \(\cY\coloneqq (\cY_t)_{t\geq 0}\) of the stochastic differential equation
\begin{equation*}
	\dd \cY_t = \sqrt{2\cY_t(1-\cY_t)}\dd B_t+(-\s \cY_t(1-\cY_t)+ \tt(1-\cY_t)-\t\nu_0 \cY_t)\dd t,\quad t\geq 0,
\end{equation*}
with \(\cY_0=y\), where \((B_t)_{t\geq0}\) is a standard Brownian motion. Furthermore, \(\cY_t\) converges in distribution, as \(t\to\infty\), to a random variable \(\cY_{\text{eq}}\) that follows \emph{Wright's distribution}, which has density
\begin{equation*}\label{Wd}
	\pi(y)=\frac{1}{C} e^{-\s y} y^{\tt-1}(1-y)^{\t\nu_0-1},\quad y\in(0,1),
\end{equation*}
where \(C\coloneqq\int_0^1 e^{-\s y} y^{\tt-1}(1-y)^{\t\nu_0-1}\dd y\) is the normalising constant. This is the unique stationary distribution of \(\cY\). 

\smallskip
As in the finite-population model, we are interested in relations between the type composition of the population at stationarity, which we identify with the present, and the ancestral population in the distant past; they are linked to the diffusion versions of the kASG and the pLD-ASG, respectively. We denote by \(\cR \coloneqq (\cR_t)_{t\geq 0}\) the kASG in the diffusion limit (see \cite{BCH18,JSP,luigi}); taking this limit in the rates \eqref{def_kASGfinite}, we see that \(\cR\) is a continuous-time Markov chain with state space \(\N_0^\D\), absorbing states \(0\) and \(\D\), and transitions from \(i>0\) to
\begin{equation*}
	i+1 \text{ at rate } i\s,\quad i-1 \text{ at rate } i(i-1+\tt), \quad \text{and } \D \text{ at rate } i\t\nu_0
\end{equation*}
in accordance with eq~\ref{special}; see Figure~\ref{Fig:R} for the transition graph. Consequently, as observed already in Section~\ref{S1}, $\mathcal R$ is an instance of the bdk $X=X^\infty$ with parameters~\eqref{special}.
\begin{figure}[t]
	\centering
	\begin{tikzpicture}[->,shorten >=1pt,auto,node distance=3cm,semithick]
		\tikzstyle{every state}=[draw=black,text=black]
		
		\node[state,minimum size=1.2cm] (0) {0};
		\node[state,minimum size=1.2cm] (1) [right=1.5cm of 0] {1};
		\node[state,minimum size=1.2cm] (2) [right=1.5cm of 1] {2};
		\node		(m_dots) [right=1cm of 2] {\(\cdots\)};
		\node[state,minimum size=1.2cm] (3) [right=1cm of m_dots] {\(i\)};
		\node[state,minimum size=1.2cm] (4) [right=1.5cm of 3] {\(i+1\)};
		\node		(r_dots) [right=1cm of 4] {\(\cdots\)};
		
		\node (d1) [above=1.5cm of 1] {\(\D\)};
		\node (d2) [above=1.5cm of 2] {\(\D\)};
		\node (d3) [above=1.5cm of 3] {\(\D\)};
		\node (d4) [above=1.5cm of 4] {\(\D\)};
		
		\node[state,white,minimum size=1.4cm] (m_down) [below = -0.3cm of m_dots] {};
		\node[state,white,minimum size=1.4cm] (m_up) [above = -0.3cm of m_dots] {};
		\node[state,white,minimum size=1.4cm] (r_down) [below = -0.3cm of r_dots] {};
		\node[state,white,minimum size=1.4cm] (r_up) [above = -0.3cm of r_dots] {};
		
		\path
		(1) edge	[cmu]				node {\(\tt\)}				(0)
		(1) edge	[ck, right]			node {\(\t\nu_0\)}			(d1)
		(2) edge	[ck, right]			node {\(i\t\nu_0\)}			(d2)
		(3) edge	[ck, right]			node {\((i+1)\t\nu_0\)}		(d3)
		(4) edge	[ck, right]			node {\(2\t\nu_0\)}			(d4)
		(2) edge	[cmu, bend left]	node {\(2(1+\tt)\)}			(1)
		(1) edge	[cl, bend left]		node {\(\s\)}				(2)
		(3) edge	[cl, bend left]		node {\(i\s\)}				(4)
		(4) edge	[cmu, bend left]	node {\((i+1)(i+\tt)\)}		(3)
		
		(m_down) edge	[out=180,in=315,cmu]	node{}		(2)
		(2) edge		[in=180,out=45,cl]		node{}		(m_up)
		(3) edge		[in=0,out=225,cmu]		node{}		(m_down)
		(m_up) edge		[in=135,out=0,cl]		node{}		(3)
		
		(r_down) edge	[out=180,in=315,cmu]	node{}		(4)
		(4) edge		[in=180,out=45,cl]		node{}		(r_up)
		;
		
	\end{tikzpicture}

\centering
	\begin{tikzpicture}[->,shorten >=1pt,auto,node distance=3cm,semithick]
		\tikzstyle{every state}=[draw=black,text=black]
		
		\node[state,minimum size=1.2cm] (1) {1};
		\node		(m_dots) [right=1cm of 1]{\(\cdots\)};
		\node[state,minimum size=1.2cm] (2) [right=1cm of m_dots] {\(i-1\)};
		\node[state,minimum size=1.2cm] (3) [right=2cm of 2] {\(i\)};
		\node[state,minimum size=1.2cm] (4) [right=2cm of 3] {\(i+1\)};
		\node		(r_dots) [right=1cm of 4] {\(\cdots\)};
		
		\node[state,white,minimum size=1.4cm] (m_down) [below = -0.3cm of m_dots] {};
		\node[state,white,minimum size=1.4cm] (m_up) [above = 1.4cm of m_dots] {};
		\node[state,white,minimum size=1.4cm] (r_down) [below = -0.3cm of r_dots] {};
		\node[state,white,minimum size=1.4cm] (r_up) [above = 1.4cm of r_dots] {};
		
		\path 
		(2) edge	[above,cl]			node {\((i-1)\s\)}			(3)
		(3) edge	[bend left,cmu]		node {\((i-1)(i+\tt)\)}		(2)
		(3) edge	[above,cl]			node {\(i\s\)}				(4)
		(4) edge	[bend left,cmu]		node {\(i(i+1+\tt)\)}		(3)
		(3) edge	[bend right=70,ck]	node {\(\t\nu_0\)}			(2)
		(4) edge	[bend right=70,ck]	node {\(\t\nu_0\)}			(2)
		(4) edge	[bend right=70,ck]	node {\(\t\nu_0\)}			(3)
		(4) edge	[ck,out=90,in=0]	node {}						(2.6,2.3)
		(3) edge	[ck,out=90,in=0]	node {}						(2.6,2.3)
		(2) edge	[ck,out=90,in=0]	node {}						(2.6,2.3)
		
		(r_up) edge		[ck,out=180,in=90]		node {}				(4)
		(r_up) edge		[ck,out=180,in=90]		node {}				(3)
		(r_up) edge		[ck,out=180,in=90]		node {}				(2)
		(m_up) edge		[ck,out=180,in=90]		node {}				(1)
		
		(1) edge		[cl,out=0,in=180]		node {}				(1.3,0) 
		(4) edge		[cl,out=0,in=180]		node {}				(11.6,0) 
		(2.7,0) edge	[cl,out=0,in=180]		node {}				(2)
		(m_down) edge	[cmu,out=180,in=315]	node {}				(1)
		(r_down) edge	[cmu,out=180,in=315]	node {}				(4)
		(2) edge		[cmu,out=225,in=0]		node {}				(m_down)
		;
		
	\end{tikzpicture}
	\caption{The transition graphs of \(\cR\) (top) and $\cL$ (bottom).\label{Fig:R}}
\end{figure}
We denote by \(\beta\deq(\beta_i)_{i\geq 0}\) the absorption probabilities in \(0\) of \(\cR\) when starting from \(i\); in analogy with the finite case~\eqref{moment_dual}, the link between the absorption probabilities and the moments of the stationary distribution of the forward process \(\cY\) is given by the diffusion version of \eqref{cor_bY_finite_eq}, which reads
\begin{equation}\label{cor_bY_diff_eq}
	\beta_i = \P(\text{\(\cR\) absorbs at \(0\)}\mid \cR_0=i) = \E[\cY_{\text{eq}}^i]= \int_{0}^1 y^i\pi(y)\dd y,\quad i\geq 0.
\end{equation}
Likewise, the sampling recursion \eqref{rec_b_finite} turns into 
\begin{equation}\label{eq:rec_b_diff}
	(i-1+\s +\t )\beta_i = \s \beta_{i+1} + (i-1+\t \nu^{}_1)\beta_{i-1},\quad i>0, 
\end{equation}
complemented by the boundary conditions \(\beta_0=1\) and \(\limi \beta_i=0\) (see \cite{BCH18,JSP}).

\smallskip
Analogously, with the parameters converging as in~\eqref{wswm}, the sequence of time-rescaled pLD-ASG's $(L^N_{Nt})_{t\ge 0}$ converges in distribution to the ``diffusion limit'' version of the pLD-ASG, which we denote by \(\cL \deq (\cL_t)_{t \geq 0}\). This is a continuous-time Markov chain with state space \(\N\) and transitions rates given by the limit of the rates \eqref{def:pLD-ASG_finite} when sped up by the factor $N$ under the condition~\eqref{wswm}. In particular, it moves from \(i>0\) to
\begin{equation*}
	i+1 \text{ at rate } i\s,\quad i-1 \text{ at rate } (i-1)(i + \tt) +\t\nu_0 \mathbbm{1}_{\{i>1\}}, \quad \text{and } j\in[i-2] \text{ at rate } \t\nu_0
\end{equation*}
(note that the rate from 1 to 0 is zero). Consequently, in accordance with the discussion at the end of Section~\ref{S1}, $\mathcal L$ is an instance of the bdc $Z=Z^\infty$ with parameters~\eqref{special}. So $\cR$ and $\cL$ indeed form a pair in the sense of Definition~\ref{def:paired}, with the transition graphs in Figure~\ref{Fig:R}. The latter is in perfect analogy with Figure~\ref{Fig:XcZ}, now for $N=\infty$ and the special parameter choice~\eqref{special}.

Due to the quadratic death rate, the process \(\cL\) is positive recurrent and thus has a unique stationary distribution, which we denote by \(\omega \coloneqq (\omega_i)_{i>0}\), with \(\alpha \deq (\alpha_i)_{i\geq 0}\) the corresponding vector of tail probabilities, so \(\cL_{\text{eq}}\), a random variable that has the stationary distribution, satisfies \(\alpha_i = \P(\cL_{\text{eq}}>i) \). This \(\alpha\) is the unique solution to \emph{Fearnhead's recursion}, which is the diffusion version of \eqref{rec_a_finite}, that is,
\begin{equation*}\label{rec_a_diff}
	(i+1+\s +\t) \alpha_i = \s \alpha_{i-1}+(i+1+\t\nu^{}_1)\alpha_{i+1},\quad i>0,
\end{equation*}
together with the boundary conditions \(\alpha_0=1\) and \(\limi \alpha_i=0\); see \cite{LKBW15}. Let now \(\gamma(y), y\in[0,1]\), be the probability that the individual at present whose progeny will take over in the population at some later time is unfit, given that the proportion of unfit individuals at present equals \(y\). We then know from a classical result by Fearnhead \cite{F02} and Taylor \cite{Taylor_07} that 
\begin{equation*}\label{eq:gamma(y)}
	\gamma(y)=1-(1-y)\sumi \alpha_i y^i,\quad y\in[0,1].
\end{equation*}
In particular, \(\gamma(y)\) is the diffusion limit version of \(g^{N}(i)\) in \eqref{eq:g(y)}.

\subsection{The connection between kASG and pLD-ASG in the diffusion limit}\label{diffusionlimit}
As in the previous subsection, we consider the counting processes $\mathcal R$ and $\mathcal L$ of the kASG and of the pLD-ASG that arise in the diffusion limit of a two-type Moran model with selection and mutation under the parameter rescaling~\eqref{wswm}. Since $\mathcal R$ and $\mathcal L$ are paired in the sense of Definition~\ref{def:paired} with $N=\infty$ and with the parameters~\eqref{special}, we conclude from Theorem~\ref{thm_general} the following
\begin{corollary}\label{thm_diff}
	In the diffusion limit, the absorption probabilities \((\beta_i)_{i\geq 0}\) of the k-ASG \(\cR\) and the tail probabilities \((\alpha_i)_{i\geq 0}\) of the pLD-ASG \(\cL\) are related via
	\begin{align*}\label{eq:diff_thm_1}
		\frac{\beta_i}{\beta_2} &= \frac{\alpha_{i-2}-\alpha_{i-1}}{\alpha_0-\alpha_1}\frac{\prod_{j=1}^{i-2} (j+1+\tt)}{\s^{i-2}}, \quad i>1,\quad \text{and} \quad
		\alpha_{i} = \frac{\beta_{i+1}-\beta_{i+2}}{\beta_1-\beta_2}\frac{\s^{i}}{\prod_{j=1}^{i} (j+\tt)}, \quad i\geq 0.
	\end{align*}
\end{corollary}

\begin{remark}\label{rem:beta_one}
	\begin{enumerate}
		\item[a)] 
		The expression for \(\alpha_1\) can be already found in \cite{Taylor_07}, in the unnumbered display below \((28)\). There \(2s=\sigma\), \(2\mu_2 = \tt\) and \(1-\widetilde p=(\beta_2-\beta_3)/(\beta_1-\beta_2)\). This expression, proved with purely analytical means, remained mysterious and was actually the original motivation for this paper.\hfill
		\item[b)]
		Combining the recursion \eqref{eq:rec_b_diff} for \(i=2\) with the above expression for \(\beta_3/\beta_2\) yields \(\beta_1/\beta_2\) in terms of \(\alpha\):
		\begin{align*}
			\frac{\beta_1}{\beta_2} = \frac{1+\s+\t}{1+\tt} - \frac{(\alpha_1-\alpha_2)(2+\tt)}{(\alpha_0-\alpha_1)(1+\tt)}.
		\end{align*}\hfill
		\item[c)]
		Since \eqref{cor_bY_diff_eq} allows for explicit integral representations of the \((\beta_i)_{i\geq 0}\), Theorem~\ref{thm_diff} provides us with one such representation for the \((\alpha_i)_{i\geq 0}\) as well:
		\begin{equation*}
			\alpha_i = \frac{\int_{0}^1 y^{i+1}(1-y)\pi(y)\dd y}{\int_{0}^1 y(1-y)\pi(y)\dd y}\frac{\s^{i}}{\prod_{j=1}^{i} (j+\tt)},\quad i\geq 0.
		\end{equation*}
		In \cite{CM19}, a different integral representation for \(\alpha\) was found in terms of hypergeometric functions by means of analytical methods. 
	\end{enumerate}
\end{remark}

\subsection{The connection between the kASG and pLD-ASG in the (finite-)\texorpdfstring{\(N\)}{} Moran model}\label{ss:app_finite}
In contrast to the diffusion limit, where Corollary~\ref{thm_diff} links the paired processes \(\cR\) and \(\cL\), the processes \(R^N\) and \(L^N\) as such cannot be paired for finite $N$, simply because \(L^N\) has state space $[N]$, while the bdc $Z^{N-1}$ that is paired with $X^N:= R^N$ has state space $[N-1]$. It turns out, however, that $Z^{N-1}$ is a simple time change of a pLD-ASG $L^{N-1}$ whose parameters are closely related to those of $R^N$.
Specifically, we have the following
\begin{lemma}
	Let $R^N$ be the k-ASG for a finite population of $N$ individuals with selection coefficient $s^N$ and mutation rate $u^N$ as in~\eqref{def_kASGfinite}, and let $Z^{N-1}$ be the process that is paired with $X^N:=R^N$ according to Definition~\ref{def:paired}. Then the process $L^{N-1}:=(Z^{N-1}_{tN/(N-1)})_{t\ge 0}$ is a pLD-ASG with parameters as in~\eqref{def:pLD-ASG_finite}, now with $N-1$ in place of $N$. The selection coefficient $s^{N-1}$ and the mutation rate $u^{N-1}$ in $L^{N-1}$ are related to the corresponding parameters in $R^N$ via
	\begin{equation}\label{parrel}
		s^{N-1} = s^N, \quad u^{N-1} = u^N \frac N{N-1}.
	\end{equation}
\end{lemma}
\begin{proof} With
	$\lambda_i^{R,N}$, $\mu_i^{R,N}$ and $\kappa^{R,N}$ as in~\eqref{def_kASGfinite}, the process $(Z^{N-1}_{tN/(N-1)})_{t\ge 0}$ is a bdc with parameters
	\begin{equation*}
		\begin{split}
			\l_i & = \myfrac{N}{N-1} \lambda_i^{R,N} = s^N \myfrac{N-i}{N-1}, \quad \mu_i = \myfrac{N}{N-1} \mu_i^{R,N} = \myfrac{{i-1}}{N-1} + u^{N}\frac N{N-1} \nu_1, \\
			\text{and } \, \kappa & = \myfrac{N}{N-1} \kappa^{R,N} = u^{N}\frac N{N-1} \nu_0.
		\end{split}
	\end{equation*}
	With the choice~\eqref{parrel}, these parameters match~\eqref{def:pLD-ASG_finite} with $N-1$ in place of $N$.
\end{proof}
Since a constant time change does not affect an equilibrium distribution, Theorem~\ref{thm_general}~A immediately gives the ratios \(b_i^N/b_2^N\) of the absorption probabilities~\eqref{eq:bnb} of \(R^N\) in terms of the equilibrium tail probabilities \(a_j^{N-1}\) of \(L^{N-1}\). Likewise, since a constant time change does not affect absorption probabilities, Theorem~\ref{thm_general}~B expresses the \(a_i^{N-1}\) in terms of the \(b_j^N\). Explicitly, we therefore get
\begin{corollary}
	\label{coro_finite}
	In the (finite-)\(N\) Moran model with selection and mutation, the sampling and tail probabilities are connected via
	\begin{align*}
		\frac{b_i^N}{b_2^N} &= \frac{a^{N-1}_{i-2}-a^{N-1}_{i-1}}{1-a^{N-1}_{1}}\prod_{j=1}^{i-2} \frac{j+1+Nu^N\nu_1}{(N-(j+1))s^N}, \quad i\in[2:N],\\
		a_i^{N-1} &= \frac{b^{N}_{i+1}-b^{N}_{i+2}}{{b^{N}_1-b^{N}_2}}\prod_{j=1}^{i} \frac{(N-(j+1))s^{N-1}}{j+ Nu^{N-1}\nu_1}, \quad i\in[N- 2]_0,
	\end{align*}
	where the parameters of the $a_i$ and $b_{j}$ are related as in \eqref{parrel} as indicated by the superscripts.
\end{corollary}

It goes without saying that, likewise, the $a_i^N$ may be expressed in terms of the $b_j^{N+1}$ as well as $s^N$ and $u^N$ --- one just has to respect the shift in population size by~1 in the parameters of the $a_i$ and $b_j$.

\begin{remark}
	Since, for \(i>0\), one has \(\alpha_i = \lim_{N\to\infty} a^N_i\) (this can be seen, for example, by combining Lemma~2 and Theorem~3 in \cite{KHB13}) and \(\beta_i = \lim_{N\to\infty} b^N_i\) (see for example \cite{BCdG24}), Corollary~\ref{thm_diff} can be obtained as a consequence of Corollary~\ref{coro_finite} in the limit \(N\to\infty\) with the proper rescaling of parameters. \hfill
\end{remark}

\begin{remark}
	As for the diffusion case, since \eqref{cor_bY_finite_eq} allows for an explicit ``discrete integral'' representation of \((b_{i}^N)_{i\in[N]_0}\), Corollary~\ref{coro_finite} provides us with one such representation for the \((a^N_i)_{i\in[N]_0}\) as well:
	\begin{equation*}
		a_i^N = \frac{\sum_{k=0}^{N+1} \frac{k^{\underline{i+1}}}{(N+1)^{\underline{i+1}}}\left( 1 - \frac{k-i-1}{N-i}\right)\widetilde\pi_k}{\sum_{k=0}^{N+1} \frac{k}{N+1} \left( 1 -\frac{k-1}{N}\right)\widetilde\pi_k}\prod_{j=1}^{i}\frac{(N+1-(j+1))s^{N}}{j+(N+1)u^N\nu_1},\quad i\in[N-1]_0,
	\end{equation*}
	where \(\widetilde \pi_k = \prod_{j=1}^{k-1}\myfrac{(N+1-j)(j+Nu^N\nu_1)}{(j+1)\left((1+s^{N})(N-j)+Nu^N\nu_0\right)}\). Once again, in \cite{CM19} a different representation of \(a^N\) was found in terms of hypergeometric functions, with the help of analytical methods. \hfill
\end{remark}

\section{Appendix: A technical lemma}
The next result provides conditions under which bdk processes absorb in $\{0,\Delta\}$ almost surely, and ensures that the probability of absorption at $0$ vanishes as the initial state tends to~$\infty$.
\begin{lemma}\label{lemma:conditions_mu_lambda}
	Let \(\mathcal{X}\coloneqq(\mathcal{X}_t)_{t\geq 0}\) be a continuous-time Markov chain on \(\N^\Delta_0\) with transitions from \(i\in\N\) to 
	\begin{equation*}
		i+1\text{ at rate } \widehat{\ell}_i,\quad i-1 \text{ at rate } \widehat{m}_i,\quad\text{and } \, \D \text{ at rate } ik
	\end{equation*}
	with \(k>0\), \(\widehat{\ell}_i,\widehat{m}_i\geq 0\) for \(i\in\N\), and \(\widehat{m}_0=0\) (so \(\mathcal{X}\) does not leave \(\N_0^\D\)). We then have:
	\begin{enumerate}
		\item If \(\widehat{\ell}_i >0\) for every \(i\in\N\) and \(\sum_{i=0}^\infty \frac{i}{\widehat{\ell}_i}=\infty\), then \(\mathcal{X}\) is non-explosive and, in particular, \(\P(\mathcal{X} \text{ absorbs in }\{0,\D\}\mid \mathcal{X}_0=n)=1\) for every \(n\in\N_0^\D\);
		\item If \(\widehat{m}_i >0\) for every \(i\in\N\) and \(\sum_{i=0}^\infty \frac{i}{\widehat{m}_i}=\infty\), then \(\lim_{n\to\infty}\P(\mathcal{X} \text{ absorbs in } 0 \mid \mathcal{X}_0=n)=0\).
	\end{enumerate}
\end{lemma}

\begin{proof}
	(1) The statement is trivially true if \(n\in\{0,\D\}\). So assume \(n\in\N\) and consider the process \(\mathcal{X}^\ell\coloneqq(\mathcal{X}^\ell_t)_{t\geq 0}\) on \(\N^\D\) with transitions from \(i\in\N\) to 
	\begin{equation*}
		i+1\text{ at rate } \widehat{\ell}_i,\quad \text{and } \, \D \text{ at rate } ik.
	\end{equation*}
	A coupling argument tells us that
    \begin{equation}\label{eq:lemma_conditions_mu_lambda_1}
		\P(\mathcal{X} \text{ explodes in finite time} \mid \mathcal{X}_0 = n)
		\leq \P(\mathcal{X}^\ell \text{ explodes in finite time}\mid \mathcal{X}^\ell_0 = n).
	\end{equation}
	In order to diverge, the embedded Markov chain of \(\mathcal{X}^\ell\) must deterministically jump from \(i\) to \(i+1\) for all \(i\geq n\), so 
	\begin{equation*}
		\P(\mathcal{X}^\ell \text{ explodes in finite time}\mid \mathcal{X}^\ell_0 = n)
		\leq \prod_{i=n}^\infty \frac{\widehat{\ell}_i}{\widehat{\ell}_i+ik}
		= \mathrm{e}^{\sum_{i=n}^\infty \log\left(\frac{\widehat{\ell}_i}{\widehat{\ell}_i+ik}\right)}.
	\end{equation*}
	We have 
	\begin{equation*}
		\sum_{i=n}^\infty \log\Bigl(\frac{\widehat{\ell}_i}{\widehat{\ell}_i+ik}\Bigr)
		= -\sum_{i=n}^\infty \log\Bigl(1+\frac{ik}{\widehat{\ell}_i}\Bigr).
	\end{equation*}
	If \(\lim_{i\to\infty} \frac{ik}{\widehat{\ell}_i}\neq 0\), the series diverges; if \(\lim_{i\to\infty} \frac{ik}{\widehat{\ell}_i}=0\), then \(\log(1+\frac{ik}{\widehat{\ell}_i})\sim \frac{ik}{\widehat{\ell}_i}\), and the series diverges under the assumption \(\sum_{i=0}^\infty \frac{i}{\widehat{\ell}_i}=\infty\). Thus, using \eqref{eq:lemma_conditions_mu_lambda_1},
	\begin{equation*}
		\P(\mathcal{X} \text{ explodes in finite time}\mid \mathcal{X}_0 = n)=0,
	\end{equation*}
	and therefore 
	\begin{equation*}
		\P(\mathcal{X} \text{ absorbs in } \{0,\D\}\mid \mathcal{X}_0 = n)=1.
	\end{equation*}

	(2) Consider the process \(\mathcal{X}^m\coloneqq(\mathcal{X}^m_t)_{t\geq 0}\) on \(\N_0^\D\) with transitions from \(i\in\N\) to 
	\begin{equation*}
		i-1\text{ at rate } \widehat{m}_i,\quad \text{and } \, \D \text{ at rate } ik.
	\end{equation*}
	A coupling argument tells us that
	\begin{equation}\label{eq:lemma_conditions_mu_lambda_2}
		\P(\mathcal{X} \text{ absorbs in } 0 \mid \mathcal{X}_0=n)
		\leq \P(\mathcal{X}^m \text{ absorbs in } 0 \mid \mathcal{X}^m_0=n).
	\end{equation}
	In order to absorb in \(0\), \(\mathcal{X}^m\) must make every transition \(i\to i-1\) for \(i\in[n]\). Thus,
	\begin{equation*}
		\P(\mathcal{X}^m \text{ absorbs in } 0 \mid \mathcal{X}^m_0=n)
		= \prod_{i=1}^n \frac{\widehat{m}_i}{\widehat{m}_i+ik}
		= \mathrm{e}^{\sum_{i=1}^n \log\left(\frac{\widehat{m}_i}{\widehat{m}_i+ik}\right)}.
	\end{equation*}
	As in part (1),
	\begin{equation*}
		\lim_{n\to\infty}\sum_{i=1}^n \log\left(\frac{\widehat{m}_i}{\widehat{m}_i+ik}\right)=-\infty
		\qquad\text{whenever}\qquad 
		\sum_{i=1}^\infty \frac{i}{\widehat{m}_i}=\infty.
	\end{equation*}
	Together with \eqref{eq:lemma_conditions_mu_lambda_2}, this proves the claim.
\end{proof}

\section*{Acknowledgements}
We are greatly indebted to two anonymous reviewers who dedicated substantial time and expertise to the manuscript and made invaluable suggestions to improve it. We would also like to thank Gerold Alsmeyer for a helpful discussion. This work was funded by the Deutsche Forschungsgemeinschaft (DFG, German Research Foundation) — Project-ID 317210226 — SFB1283.

\end{document}